\documentclass[10pt]{amsart}
\usepackage{indentfirst,latexsym,bm}
\usepackage{amsfonts}
\usepackage{amssymb}
\usepackage{amsmath}
\usepackage{hyperref}
\usepackage{dsfont}
\usepackage{leftidx}
\usepackage{amsthm}
\usepackage{geometry}
\usepackage{amsmath,amscd}
\usepackage{tikz}
\usepackage[all,cmtip]{xy}

\usepackage{latexsym,amsmath,amssymb,amsfonts}  
\usepackage{color}
\usepackage{color,xcolor}
\usepackage{colordvi}
\usepackage{xspace}
      
\usepackage{graphicx}
\usepackage{pstricks}
\usepackage{pgf}
\usepackage{tikz}
\usepackage{xkeyval}
\usepackage{tikz}

\usetikzlibrary{mindmap,trees,shadows}

\begin{document}
 \thispagestyle{empty}

\title[A Noetherian Hopf algebra is affine iff its Hopf coradical is affine]
{A Noetherian Hopf algebra is affine iff its Hopf coradical is affine}

\author{
{Huan Jia $^{1}$}
and 
{Yinhuo Zhang $^{2}$}}

\address{$^{1}$ Department of Mathematics, Suqian University, Suqian City 223800, Jiangsu Province, China}

\address{$^{2}$ Department of Mathematics and Statistics,
University of Hasselt, Universitaire Campus, 3590 Diepenbeek, Belgium}

\address{Huan Jia $^{1}$}
\email{huan.jia@squ.edu.cn}

\address{Yinhuo Zhang $^{2}$}
\email{yinhuo.zhang@uhasselt.be}

\date{\today}
\maketitle

\newtheorem{theorem}{Theorem}[section]
\newtheorem{proposition}[theorem]{Proposition}
\newtheorem{lemma}[theorem]{Lemma}
\newtheorem{corollary}[theorem]{Corollary}
\theoremstyle{definition}
\newtheorem{definition}[theorem]{Definition}
\newtheorem{example}[theorem]{Example}
\newtheorem{remark}[theorem]{Remark}

\newcommand{\K}{\mathds{k}}

\newcommand{\lOr}{\langle O \rangle}
\newcommand{\X}{\langle X \rangle}
\newcommand{\Y}{\langle Y \rangle}
\newcommand{\Xo}{\langle X_{0} \rangle}
\newcommand{\Xp}{\langle X_{+} \rangle}
\newcommand{\Yo}{\langle Y_{0} \rangle}
\newcommand{\Yp}{\langle Y_{+} \rangle}

\newcommand{\id}{\operatorname{id}}
\newcommand{\im}{\operatorname{Im}}
\newcommand{\rank}{\operatorname{rank}}
\newcommand{\corad}{\operatorname{corad}}
\newcommand{\gr}{\operatorname{gr}}
\newcommand{\s}{\operatorname{s}}

\newcommand{\lex}{\operatorname{lex}}
\newcommand{\LW}{\operatorname{LW}}
\newcommand{\mlex}{\operatorname{mlex}}
\newcommand{\mrf}{\operatorname{mrf}}
\newcommand{\irr}{\operatorname{irr}}

\theoremstyle{plain}
\newcounter{maint}
\renewcommand{\themaint}{\Alph{maint}}
\newtheorem{mainthm}[maint]{Theorem}

\theoremstyle{plain}
\newtheorem*{proofthma}{Proof of Theorem A}
\newtheorem*{proofthmb}{Proof of Theorem B}

\begin{abstract}
Let $\K$ be an arbitrary field. Extending the structural frameworks established in \cite{JZ2025-2}, this paper introduces novel techniques utilizing mirror reduction orders, prime words, and the generalized lifting methodology. We establish a definitive necessary and sufficient criterion for the affineness of Noetherian Hopf algebras, thereby providing a significant advancement toward resolving the long-standing Wu--Zhang question \cite{WZ2003}. Specifically, we prove that a left or right Noetherian Hopf algebra over $\K$ is affine if and only if its Hopf coradical is affine. 
This characterization fundamentally reduces the verification to the first filtration step, yielding a criterion that is structurally transparent and highly operational. 

To establish necessity, we provide an essential intrinsic result demonstrating that the Hopf coradical of an affine Hopf algebra inherits the property of being affine. Furthermore, as direct applications of this equivalence, we prove that a left or right Noetherian Hopf algebra $H$ is affine provided that its coradical $H_{(0)}$ forms a subalgebra (the dual Chevalley property), its coradical $H_{(0)}$ is cocommutative, or its Hopf coradical $H_{[0]}$ is commutative.

\end{abstract}

\section*{Introduction}

An algebra over $\K$ is called \textit{affine} if it is finitely generated as an algebra. By the Hilbert basis theorem, every commutative affine algebra is Noetherian, but the converse fails in general (e.g., the formal power series algebra $\K[[x]]$). For Hopf algebras, however, Molnar \cite{Mo1975} established that any commutative or cocommutative Noetherian Hopf algebra is affine.

In 2003, Q.-S.~Wu and J.~J.~Zhang proved that every affine Noetherian polynomial identity (PI) Hopf algebra is Artin--Schelter Gorenstein, and in doing so, they formulated the well-known affineness question for general Noetherian Hopf algebras (see also \cite{Br2007, BG2014, Go2013, A2023}):

\medskip
\noindent
\textbf{Question \cite{WZ2003}}.
\textit{Is every Noetherian Hopf algebra over $\K$  affine?}
\medskip

The general problem has remained widely open, even for PI Hopf algebras, which share close geometric and algebraic ties to the commutative setting (see \cite{BG2014, Go2013}). In the graded framework, Jia and Zhang \cite{JZ2025-1} recently showed that a Noetherian Hopf algebra that is graded as an algebra is affine if and only if its degree-zero homogeneous component is affine.

In the non-graded framework, Goodearl and Zhang \cite{GZ2017} proved that a locally affine Noetherian Hopf algebra over an algebraically closed field of characteristic zero is affine, provided it is faithfully flat over all its Hopf subalgebras. As a direct consequence, any Noetherian pointed Hopf algebra over such a field is affine. Subsequently, Jia and Zhang \cite{JZ2025-2} introduced the machinery of reduction orders, reduction-factorization, and prime words, successfully establishing that every Noetherian pointed Hopf algebra over an arbitrary field is affine. 
In a recent survey \cite{Br2026}, K.~Brown pointed out that the original constraints in \cite{GZ2017} requiring the base field to be algebraically closed and of characteristic zero could indeed be dropped. 

However, the powerful field-independent machinery introduced in \cite{JZ2025-2} extends naturally to the much broader framework of the present paper, after appropriate technical modifications. This extension underscores the robustness and general utility of the underlying rewriting principles in non-commutative algebra and the structural analysis of infinite-dimensional Hopf algebras.

In the present paper, we investigate the Wu--Zhang affineness question in the wider setting where the Hopf coradical of a Noetherian Hopf algebra is affine. The concept of the Hopf coradical, introduced by Andruskiewitsch and Cuadra \cite{AC2013}, generalizes the celebrated lifting method of Andruskiewitsch and Schneider \cite{AS1998}, which serves as a cornerstone in the classification theory of both finite- and infinite-dimensional Hopf algebras. 

By employing mirror-reduction orderings and prime-word rewriting machinery, and the generalized lifting technique, we prove that a left or right Noetherian Hopf algebra over a field is affine if and only if its Hopf coradical is affine. A vital component in demonstrating the necessity of this condition is an intrinsic structural result showing that the Hopf coradical of an affine Hopf algebra inherits the property of being affine.

The manuscript is organized as follows. Let $(X, <_{X})$ be a well-ordered set, let $\X$ denote the free monoid generated by $X$, and let $\K\X$ be the free associative algebra on $X$.

In Section~\ref{sec:1}, we review the required background material on Hopf algebras, the generalized lifting method, and the (mirror) lexicographic order.
In Section~\ref{sec:2}, we introduce the mirror-reduction order on $\X$ along with its associated reduction-factorization theory, thereby refining and generalizing the structural techniques originally developed in \cite{JZ2025-2}.

In Section~\ref{sec:3}, we analyze the properties of irreducible words with respect to the chosen reduction order. We show that if an augmented algebra is left Noetherian, the underlying set of its irreducible letters must be finite (Lemma~\ref{lem:O_I-finite}).
In Section~\ref{sec:4}, given a Hopf algebra $H$ over $\K$ whose Hopf coradical $H_{[0]}$ is a Hopf subalgebra, we construct a generating set $X$ extending a given base set $O$ such that the coproduct $\Delta_H$ lifts to a skew-triangular comultiplication on the free algebra $\K\X$ (Definition~\ref{def:comultiplication}, Lemma~\ref{lem:X-generates-H}). By applying mirror reduction order, reduction-factorization, and prime words, we completely determine the comultiplication formulas for arbitrary words (Lemmas~\ref{lem:comultiplication-prime-words}--\ref{lem:comultiplication-words-2}).

In Section~\ref{sec:5}, we establish two fundamental reduction results governing prime words (Lemmas~\ref{lem:beta-qbeta-h-I}--\ref{lem:beta-u-I}). We then prove that every reducible letter of positive degree can be expanded as a linear combination of products consisting of strictly smaller letters (Lemma~\ref{lem:rl-sum-smaller-irrls}). As a consequence, every non-empty word admits a decomposition in terms of degree-zero letters and irreducible positive-degree letters (Theorem~\ref{thm:rl-sum-lower-irrls}). Combining this rewriting framework with Lemma~\ref{lem:O_I-finite}, we demonstrate that a left or right Noetherian Hopf algebra is affine whenever its Hopf coradical is affine (Theorem~\ref{thm:main-theorem2}). More generally, this affineness property holds if the coradical is contained in any affine Hopf subalgebra (Remark~\ref{rmk:more-result}).

In Section~\ref{sec:6}, we turn our attention to the converse direction and prove the following core result:

\medskip
\noindent
\textbf{Theorem A.} (Theorem~\ref{thm:main-theorem3})\ \ 
\textit{The Hopf coradical of an affine Hopf algebra over $\K$ is affine.}
\medskip

By consolidating Theorem~\ref{thm:main-theorem2} with Theorem~A, we obtain a new, equivalent formulation of the Wu--Zhang question that shifts the focus onto the first filtration step:

\medskip
\noindent
\textbf{Theorem B.} (Theorem~\ref{thm:main-theorem4})\ \ 
\textit{Let $H$ be a left or right Noetherian Hopf algebra over $\K$. Then the following conditions are equivalent:}
\begin{itemize}
    \item[$\mathrm{(a)}$] \textit{H is affine;}
    \item[$\mathrm{(b)}$] \textit{The coradical $H_{(0)}$ is contained in an affine Hopf subalgebra of $H$.}
    \item[$\mathrm{(c)}$] \textit{The Hopf coradical $H_{[0]}$ is affine;}
\end{itemize}
\textit{In this case, any Hopf subalgebra of $H$ containing $H_{(0)}$ is also affine.}
\medskip

Consequently, the remaining step in resolving the global affineness question is to establish the affineness of the Hopf coradical for arbitrary Noetherian Hopf algebras, or equivalently, to determine whether a Noetherian Hopf algebra that is generated as an algebra by its coradical is necessarily affine.

In Section~\ref{sec:7}, we explore the structural boundaries of this problem under further representation-theoretic conditions:

\medskip
\noindent
\textbf{Theorem C.} (Theorem~\ref{thm:when-Hopfalg-is-affine})\ \ 
\textit{Let $H$ be a left or right Noetherian Hopf algebra over $\K$. Suppose that the Hopf coradical $H_{[0]}$ is locally affine and faithfully flat over all its Hopf subalgebras. Then $H_{[0]}$ is affine, and hence $H$ is affine.}
\medskip

We note that these structural hypotheses are readily satisfied by many celebrated families of Hopf algebras (Remark~\ref{rmk:conditions-examples}). In particular, as explicit applications of this framework, we confirm the affineness of left or right Noetherian Hopf algebras that satisfy the dual Chevalley property, whose coradicals are cocommutative, or whose Hopf coradicals are commutative.

Finally, several conceptual techniques utilized in this work are inspired by the foundational frameworks in \cite{ZSL2020, Ji2023, Kh1999, Ro1999, Uf2004}, particularly concerning irreducible words defined with respect to graded lexicographic orders, and the utility of comultiplications of free algebras in Poincar\'e--Birkhoff--Witt-type constructions.

\section{Preliminaries}\label{sec:1}

Throughout, let $\mathbb{N}$ denote the set of natural numbers.

\subsection{Hopf algebras}
\quad

We refer to \cite{Ra2012,Sw1969,Mo1993} for basic concepts concerning coalgebras and Hopf algebras.
For a coalgebra $C$ over $\K$, let $\Delta_{C}$ denote its coproduct and $\epsilon_{C}$ its counit.
We use Sweedler's notation $\Delta_{C}(c) = \sum c_{(1)} \otimes c_{(2)}$ for the coproduct.
The \textit{coradical filtration} $\{C_{(n)}\}_{n=0}^{\infty}$ of $C$ is defined as follows:
\begin{itemize}
\item the \textit{coradical} $C_{(0)}$ is the sum of all simple subcoalgebras of $C$;
\item $C_{(n)} := C_{(n-1)} \bigwedge C_{(0)}= \Delta_{C}^{-1}(C_{(n-1)} \otimes C + C \otimes C_{(0)})$ for $n\ge 1$.
\end{itemize}
This forms a coalgebra filtration of $C$. 
For a Hopf algebra $H$ over $\K$, the coalgebra filtration $\{H_{(n)}\}_{n\ge 0}$ is a Hopf algebra filtration of $H$ if and only if the coradical $H_{(0)}$ is a Hopf subalgebra of $H$.

In general, a Hopf algebra need not satisfy the dual Chevalley property; that is, its coradical is not always a subalgebra. To study such Hopf algebras, Andruskiewitsch and Cuadra introduced the generalized lifting method \cite{AC2013}.
The \textit{standard filtration} $\{H_{[n]}\}_{n\ge 0}$ of a Hopf algebra $H$ over $\K$ is defined as follows:
\begin{itemize}
\item the \textit{Hopf coradical} $H_{[0]}$ is the subalgebra of $H$ generated by the coradical $H_{(0)}$;
\item $H_{[n]}:= H_{[n-1]} \bigwedge H_{[0]}$ for $n \ge 1$.
\end{itemize}
Since $H_{[0]}$ is a sub-bialgebra of $H$, this filtration is a bialgebra filtration of $H$.
When $H_{[0]} = H_{(0)}$, the standard filtration of $H$ recovers the coradical filtration of $H$. 
Let $S_{H}$ denote the antipode of $H$.
Assume $S_{H}(H_{[0]}) \subseteq H_{[0]}$ (this holds if $S_{H}|_{H_{(0)}}$ is injective). Then the bialgebra filtration $\{H_{[n]}\}_{n\ge 0}$ is a Hopf algebra filtration of $H$.
Denote by $$\gr_{\s} H := \bigoplus_{n\ge 0} H_{[n]} / H_{[n-1]}$$ the associated graded Hopf algebra with respect to $\{H_{[n]}\}_{n\ge 0}$, and set $(\gr_{\s} H)_{n}:= H_{[n]}/ H_{[n-1]}$.

For a Hopf algebra $H$ over $\K$, 
${_{H}^{H}\mathcal{YD}}$ is the category of left Yetter-Drinfel'd modules, 
whose objects are triples $(M, \cdot, \rho)$, where $(M, \cdot)$ is a left $H$-module and  $(M, \rho)$ is a left $H$-comodule (write $\rho(m) =\sum m_{(-1)} \otimes m_{(0)} $), satisfying the compatibility condition:
\begin{align*}
h_{(1)} m_{(-1)} \otimes h_{(2)} \cdot m_{(2)} = (h_{(1)} \cdot m)_{(-1)} h_{(2)} \otimes (h_{(1)} \cdot m)_{(0)}, \quad h\in H,~ m\in M.
\end{align*}
Let $M, N \in {_{H}^{H}\mathcal{YD}}$. The module and the comodule of $M \otimes N$ in ${_{H}^{H}\mathcal{YD}}$ are given by: 
\begin{align*}
h \cdot (m \otimes n) &= \sum h_{(1)} \cdot m \otimes h_{(2)} \cdot n, \quad h \in H, ~m\in M, ~n \in N,\\
\rho( m \otimes n ) &= \sum m_{(-1)} n_{(-1)} \otimes m_{(0)} \otimes n_{(0)}, \quad m\in M, ~n \in N.
\end{align*}
The category ${_{H}^{H}\mathcal{YD}}$ is prebraided monoidal: the prebraiding $c_{M,N}: M \otimes N \rightarrow N \otimes M$ for $M,N \in {_{H}^{H}\mathcal{YD}}$ is given by 
\begin{align*}
c_{M,N} (m \otimes n) = \sum m_{(-1)} \cdot n \otimes m_{(0)}, \quad m \in M, ~n \in N. 
\end{align*}

For a Hopf algebra whose Hopf coradical is a Hopf subalgebra, the graded associated Hopf algebra with respect to the standard filtration has the Radford biproduct decomposition \cite{Ra1985,Ra2012}:
\begin{lemma}\cite{AC2013} \label{lem:gr[c]H-Radford-biproduct}
Let $H$ be a Hopf algebra over $\K$ such that $S_{H}(H_{[0]}) \subseteq H_{[0]}$. Then there exists an $\mathbb{N}$-graded connected Hopf algebra $B  = \bigoplus_{n\ge 0} B_{n}$ in ${_{H_{[0]}}^{H_{[0]}}\mathcal{YD}}$ such that as Hopf algebras,
\begin{align*}
\gr_{\s}  H \cong B \sharp H_{[0]},
\end{align*}
where $B = (\gr_{\s} H)^{co~\pi}$, $B_{n} = B \cap (\gr_{\s} H)_{n}$ and $(\gr_{\s} H)_{n} = B_{n}H_{[0]}$ for every $n\ge 0$.
\end{lemma}

The antipode of a Hopf algebra is often injective, and in many case even bijective. 
For any Noetherian Hopf algebra, Skryabin proved the injectivity of the antipode, and conjectured that it is bijective.
\begin{lemma}\cite{Sk2006}\label{lem:Sk2006}
The antipode of a right or left Noetherian Hopf algebra is injective.
\end{lemma}
Evidently, the Hopf coradical of a Noetherian Hopf algebra is a Hopf subalgebra.

\subsection{Mirror lexicographic order on words}
\quad

For a set $X$, $\X$ (resp. $\X^{+}$, $\K\X$) is the free monoid (resp. free semigroup, free algebra) generated by $X$. 
The empty word in $\X$ is denoted by $1$. 
Clearly, $\X = \X^{+} \cup \{1\}$, and $\K\X^{+}$ is the subspace of $\K\X$ spanned by $\X^{+}$. Write $X^{*} := X \cup \{1\}$.

A word $\beta\in \X$ is called a \textit{factor} of $\alpha\in \X$ if there exist $\gamma,\eta\in \X$ such that $\alpha=\gamma\beta\eta$. $\beta$ is called a \textit{prefix} (resp. \textit{suffix}) of $\alpha$ if $\gamma=1$ (resp. $\eta=1$). If $\beta\neq \alpha$, then $\beta$ is called a \textit{proper} factor of $\alpha$.

For a word $\alpha= x_{1} \cdots x_{n} \in \X$ with letters $x_{1},\ldots,x_{n}\in X$,  define
\begin{itemize}
\item  $\alpha(y)$ as the number of occurrences of a letter $y\in X$ in $\alpha$, that is, \\
 $\alpha(y):=\# \{i \mid x_{i}=y,~1\le i\le n\}$;
\item  $r(\alpha)$ as the rightmost letter of  $\alpha$, i.e.  $r(\alpha)=x_{n}$; 
\item  $|\alpha|$ as the \textit{length} of  $\alpha$, i.e. $|\alpha|=n$. 
\end{itemize}
For each $m\ge 1$, write $X^{m}:= \{\alpha \in \X \mid |\alpha| = m\}$.

Given a well-ordered set $(X, <_{X})$. The \textit{lexicographic order} $<_{\lex_{X}}$ on $\X$ is defined as follows (see \cite{Lo1997}). For words $\alpha,\beta\in \X$, 
\begin{align*}
\alpha <_{\lex_{X}} \beta \Longleftrightarrow 
\left\{\begin{array}{l}
\alpha \text { is a proper prefix of } \beta, \text { or } \\
\alpha=\gamma x \eta, ~\beta=\gamma y \zeta \text { with } x, y \in X, ~x<_{X} y \text{ and } \gamma, \eta, \zeta \in\langle X\rangle.
\end{array}\right.
\end{align*}
Define a \textit{mirror lexicographic order} $<_{\mlex_{X}}$ on $\X$ as follows. For words $\alpha,\beta\in \X$, 
\begin{align*}
\alpha <_{\mlex_{X}} \beta \Longleftrightarrow 
\left\{\begin{array}{l}
\alpha \text { is a proper suffix of } \beta, \text { or } \\
\alpha =\eta x \gamma, ~ \beta=  \zeta y \gamma \text { with } x, y \in X, ~x <_{X} y \text{ and } \gamma, \eta, \zeta \in\X.
\end{array}\right.
\end{align*}

Note that the (mirror) lexicographic order is a total order on $\X$ but not a well order.

\begin{example}
Let $X=\{x,y\}$ with $x <_{X} y$. By definition, we have
\begin{align*}
& x <_{\lex_{X}} x^{2} <_{\lex_{X}} x^{3}  <_{\lex_{X}} xy  <_{\lex_{X}} y <_{\lex_{X}} yx <_{\lex_{X}} y^{2}; \\
& x <_{\mlex_{X}} x^{2}  <_{\mlex_{X}} x^{3} <_{\mlex_{X}} yx  <_{\mlex_{X}} y <_{\mlex_{X}} xy <_{\mlex_{X}} y^{2}.
\end{align*}
\end{example}

\section{Mirror Reduction Order, Reduction Factorization, and Prime Words}\label{sec:2}

In this section, our analysis builds directly on the techniques, results, and proof arguments presented in \cite{JZ2025-2}, which we further modify here. Therefore, we omit restating those proofs unless specifically necessary for clarity.

Let $O$ be a set equipped with a map $t_{O}: O \rightarrow \mathbb{N}$. For each $m \ge 0$, we denote the following subsets:
\begin{align*}
O_{m} &:= \{y \in O \mid t_{O}(y) = m\}, \qquad O_{+} := \{y \in O \mid t_{O}(y) \ge 1\}.
\end{align*}
Recall that $1$ denotes the empty word. For each $y \in O_{+}$, we adopt the convention $y_{1} = y$, and define the following pairwise disjoint sets of formal symbols:
\begin{align*}
C_{y} &:= \{y_{a} \mid a \in O^*_{0} \}, \qquad X := O_{0} \cup \left(\bigcup_{y\in O_{+}} C_{y}\right).
\end{align*}
Evidently, $O \subseteq X$. We define a map $C: X \rightarrow O$ by 
\begin{align*}
C(a) = a, \quad C(y_{b}) = C(y) = y, \qquad \text{for } a, b \in O_{0},\ y \in O_{+}.
\end{align*}
We refer to $O$ as the \textit{original set} of $X$, and call $y$ the \textit{original element} of $y_{b}$. The map $C$ naturally lifts to a monoid map $C:\X \rightarrow \lOr$.

Next, we define a map $t_{X}: X \rightarrow \mathbb{N}$ by 
\begin{align*}
t_{X}(a) = t_{O}(a), \quad t_{X}(y_{b}) = t_{X}(y) = t_{O}(y), \qquad \text{for } a, b \in O_{0},\  y \in O_{+}. 
\end{align*}
For each $m \ge 0$, we denote:
\begin{align*}
X_{m} := \{x \in X \mid t_{X}(x) = m\}, \qquad X_{+} := \{x \in X \mid t_{X}(x) \ge 1\}.
\end{align*}
It follows that $X_{0} = O_{0}$, $O_{+} \subseteq \bigcup_{y\in O_{+}} C_{y} = X_{+}$, and $\lOr, \Xo, \Xp \subseteq \X$.

Let $(O, <_{O})$ be a well-ordered set. For each $z \in O_{+}$, we define an order $<_{C_{z}}$ on $C_{z}$ as follows. For any $x, y \in C_{z}$,
\begin{align*}
x <_{C_{z}} y 
\Longleftrightarrow 
\left\{\begin{array}{l}
 x = z \text{ and } y = z_{a} \text{ for some } a \in O_{0}, \text{ or } 
 \\
 x = z_{b} \text{ and } y = z_{c} \text{ for some } b,c \in O_{0} \text{ with } b <_{O} c.
\end{array}\right.
\end{align*}
It is straightforward to verify that $(C_{z}, <_{C_{z}})$ is well-ordered.

Using this, we define an order $<_{X}$ on $X$ as follows. For any $x, y \in X$,
\begin{align*}
x <_{X} y 
\Longleftrightarrow 
\left\{\begin{array}{l}
 C(x) <_{O} C(y), \text{ or } 
 \\
 C(x) = C(y) = z \in O_{+} \text{ and } x <_{C_{z}} y.
\end{array}\right.
\end{align*}
Consequently, $<_{X}$ is a total order on $X$.

We introduce the following reduction orders on words.
\begin{definition}\label{def:reduction-order}
The (mirror) \textit{reduction order} $<_{\lOr}$ on $\lOr$ is defined as follows: for any non-empty words $\alpha, \beta \in \lOr$, we set $1 <_{\lOr} \alpha$, and 
\begin{align*}
\alpha <_{\lOr} \beta \Longleftrightarrow 
& \left\{\begin{array}{l}
r(\alpha) <_{O} r(\beta), \text{ or } \\
r(\alpha) = r(\beta) \text{ and } 
\left\{\begin{array}{l}
|\alpha| < |\beta|, \text{ or } \\
|\alpha| = |\beta| \text{ and } \alpha <_{\mlex_O} \beta.
\end{array}\right.
\end{array}\right.
\end{align*}
The (mirror) \textit{reduction order} $<_{\X}$ on $\X$ is defined as follows: for any words $\alpha, \beta \in \X$, 
\begin{align*}
\alpha <_{\X} \beta \Longleftrightarrow 
& \left\{\begin{array}{l}
C(\alpha) <_{\lOr} C(\beta), \text{ or } \\
C(\alpha) = C(\beta) \text{ and } \alpha <_{\mlex_{X}} \beta.
\end{array}\right.
\end{align*}
\end{definition}

\begin{example}\label{eg:reduction-order}
Let $O = \{a, x, y\}$ with $O_{0} = \{a\}$, $O_{+} = \{x,y\}$, ordered as $a <_{O} x <_{O} y$. 
By definition, we have $C_{x} = \{x, x_{a}\}$, $C_{y} = \{y, y_{a}\}$, and $X = \{a, x, x_{a}, y, y_{a}\}$, where
$$a <_{X} x <_{X} x_{a} <_{X} y <_{X} y_{a}.$$
Moreover, the induced order yields:
\begin{align*}
x <_{\X} x_{a} <_{\X} x^{2} <_{\X} x_{a}x <_{\X} xx_{a} <_{\X} x_{a}^{2}.
\end{align*}
\end{example}

We show that the following constructed sets preserve the well-ordering property.
\begin{lemma}\label{lem:well-orders}
The ordered sets $(X, <_{X})$, $(\lOr, <_{\lOr})$, and $(\X, <_{\X})$ are well-ordered.
\end{lemma}

\begin{proof}
Let $S$ be a non-empty subset of $X$. Note that $\emptyset \neq C(S) \subseteq O$ and $(O, <_{O})$ is well-ordered. Thus, $C(S)$ possesses a unique least element with respect to $<_{O}$, which we denote by $a$. If $a \in O_{0}$, then $a$ is the least element of $S$ with respect to $<_{X}$. Alternatively, assume that $a \in O_{+}$, and define the fiber 
\begin{align*}
S_{a} := \{s \in S \mid C(s) = a\}.
\end{align*} 
Since $\emptyset \neq S_{a} \subseteq C_{a}$ and $(C_{a}, <_{C_{a}})$ is well-ordered, $S_{a}$ contains a least element with respect to $<_{C_{a}}$, say $b$. By definition, $b$ constitutes the least element of $S$ with respect to $<_{X}$. Therefore, $(X, <_{X})$ is well-ordered.

Next, let $T$ be a non-empty subset of $\lOr$, and define the set of rightmost letters:
\begin{align*}
T_{r} := \{r(\alpha) \mid \alpha \in T\}.
\end{align*}
Given that $\emptyset \neq T_{r} \subseteq O$ and $(O, <_{O})$ is well-ordered, $T_{r}$ has a least element with respect to $<_{O}$, say $c$. Let 
\begin{align*}
T_{r,c} := \{\alpha \in T \mid r(\alpha) = c\}.
\end{align*}
Among the words in $T_{r,c}$, choose those with minimal length $m$, and set 
\begin{align*}
T_{r,c,m} := \{\alpha \in T \mid r(\alpha) = c, ~|\alpha| = m\}.
\end{align*} 
Observe that $\emptyset \neq T_{r,c,m} \subseteq O^{m}$ and $(O^{m}, <_{\mlex_{O}})$ is well-ordered. Thus, $T_{r,c,m}$ has a unique least element with respect to $<_{\mlex_{O}}$, which can be expressed as $\beta c$ for some word $\beta$. By definition, $\beta c$ is the least element of $T$ with respect to $<_{\lOr}$, and hence $(\lOr, <_{\lOr})$ is well-ordered.

Finally, let $U$ be a non-empty subset of $\X$. Since $\emptyset \neq C(U) \subseteq \lOr$ and $(\lOr, <_{\lOr})$ is well-ordered, $C(U)$ has a least element with respect to $<_{\lOr}$, say $\gamma$. Let 
\begin{align*}
U_{\gamma} := \{\alpha \in \X \mid C(\alpha) = \gamma\}.
\end{align*}
Note that $\emptyset \neq U_{\gamma} \subseteq X^{|\gamma|}$. Since $(X, <_{X})$ is well-ordered, the Cartesian power $(X^{|\gamma|}, <_{\mlex_{X}})$ is also well-ordered. Then $U_{\gamma}$ contains a unique least element with respect to $<_{\mlex_{X}}$, say $\eta$. By definition, $\eta$ is the least element of $U$ with respect to $<_{\X}$. Therefore, $(\X, <_{\X})$ is well-ordered.
\end{proof}

The following lemma establishes that these defined orders are mutually compatible.

\begin{lemma}\label{lem:compatibility-orders}
The following assertions hold:
\begin{itemize}
\item [(a)] $<_{\lOr}$ is compatible with $<_{O}$ on $O$; that is, for any letters $x, y \in O$, $x <_{\lOr} y$ if and only if $x <_{O} y$.
\item [(b)] $<_{X}$ is compatible with $<_{O}$ on $O$; that is, for any letters $x, y \in O$, $x <_{X} y$ if and only if $x <_{O} y$.
\item [(c)] $<_{\X}$ is compatible with $<_{X}$ on $X$; that is, for any letters $x, y \in X$, $x <_{\X} y$ if and only if $x <_{X} y$.
\item [(d)] $<_{\X}$ is compatible with $<_{\lOr}$ on $\lOr$; that is, for any words $\alpha, \beta \in \lOr$, $\alpha <_{\X} \beta$ if and only if $\alpha <_{\lOr} \beta$.
\end{itemize}
\end{lemma}

Applying Lemma~\ref{lem:compatibility-orders} to words yields the following immediate consequence:

\begin{corollary}\label{cor:well-orders-compatibility-2}
Let $\alpha$ be a non-empty word in $\X$ and let $y$ be a letter in $O$. Then the following conditions are equivalent: 
\begin{itemize}
    \item [(a)] $\alpha <_{\X} y$;
    \item [(b)] $r(\alpha) <_{X} y$;
    \item [(c)] $r(C(\alpha)) = C(r(\alpha)) <_{O} y$;
    \item [(d)] $C(\alpha) <_{\lOr} y$.
\end{itemize}
\end{corollary}

The following lemma demonstrates that the reduction orders $<_{\lOr}$ and $<_{\X}$ possess the property of left-compatibility:

\begin{lemma}\label{lem:ac<bc}
Let $\alpha, \beta$ be non-empty words in $\lOr$ $($resp. $\X$$)$ satisfying $\alpha <_{\lOr} \beta$ $($resp. $\alpha <_{\X} \beta$$)$. Then $\gamma\alpha <_{\lOr} \gamma\beta$ for any word $\gamma$ in $\lOr$ $($resp. $\X$$)$.
\end{lemma}

\begin{remark}
In general, the reduction order is not right-compatible. That is, given $\alpha <_{\X} \beta$ (resp. $\alpha <_{\lOr} \beta$), the inequality $\alpha\gamma <_{\X} \beta\gamma$ (resp. $\alpha\gamma <_{\lOr} \beta\gamma$) does not necessarily hold for an arbitrary word $\gamma \in \X$ (resp. $\gamma \in \lOr$). 
For instance, Example~\ref{eg:reduction-order} illustrates that $x^{2} <_{\X} y$. However, we have $x^{2} y_{a} >_{\X} yy_{a}$ because $C(x^{2}y_{a}) = x^{2}y >_{\lOr} y^{2} = C(yy_{a})$.
\end{remark}

By virtue of Lemma~\ref{lem:compatibility-orders}, we obtain the following minimality property for letters in $O_{+}$.

\begin{lemma}\label{lem:x-least}
Let $x \in O_{+}$ and define $R_{x} := \{\alpha \in \X \mid C(r(\alpha)) = x\}$. Then $x$ is the unique least element of the set $R_{x}$ with respect to $<_{\X}$. Moreover, for a letter $\ell \in X_{+}$, the following conditions are equivalent:
\begin{itemize}
\item [(a)] $\ell$ is the least non-empty word of $\Xp$ with respect to $<_{\X}$;
\item [(b)] $\ell$ is the least non-empty word of $\langle O_{+} \rangle$ with respect to $<_{\lOr}$; 
\item [(c)] $\ell$ is the least letter of $X_{+}$ with respect to $<_{X}$; 
\item [(d)] $\ell$ is the least letter of $O_{+}$ with respect to $<_{O}$.
\end{itemize}
\end{lemma}

We now introduce the canonical factorization of a word with respect to the reduction order $<_{\X}$.

\begin{definition}
Let $\alpha$ be a non-empty word in $\X$. The (mirror) \textit{reduction-factorization} of $\alpha$ (or \textit{mr-factorization} for short), denoted by $\mathrm{mrf}(\alpha) = (\alpha_{L}, \alpha_{R})$, is the unique decomposition 
$$\alpha = \alpha_{L}\alpha_{R}$$ 
where $\alpha_{L}$ is the greatest prefix of $\alpha$ with respect to the reduction order $<_{\X}$. 

A non-empty word $\alpha \in \X$ is said to be \textit{prime} if its decomposition satisfies $\alpha_{L} = \alpha$ (or equivalently, $\alpha_{R}=1$). We denote the set of all prime words in $\X$ by $\X_{p}$.  
\end{definition}

By definition, every letter $x \in X$ is trivially a prime word since $x_{L} = x$ and $x_{R} = 1$.

\begin{example}\label{eg1:xR}
Let $O = \{x,y,z\}$ with $O_{0} = \emptyset$ and $x <_{O} y <_{O} z$. Then we have $C_{x}=\{x\}$, $C_{y}=\{y\}$, $C_{z}=\{z\}$, and $X=\{x,y,z\}$, ordered as $x <_{X} y <_{X} z$. Observe that 
$$(x^{3})_{L} = x^{3}, \quad (yx)_{L} = y, \quad \text{and} \quad (xy)_{L} = xy.$$
This indicates that the mr-factorization of a word is intrinsically related to the greatest letter it contains with respect to $<_{X}$.
Moreover, note that $yxx >_{\X} zx$, whereas $(yxx)_{L} = y <_{\X} z = (zx)_{L}$. This demonstrates that a word with a higher reduction order may possess a prefix with a lower reduction order. 
\end{example}

\begin{example}
Recall from Example~\ref{eg:reduction-order} that
\begin{align*}
C(y_{a}x) = yx <_{\lOr} y = C(y_{a}) <_{\lOr} yxy = C(y_{a}xy).
\end{align*} 
Hence, it follows that 
\begin{align*}
y_{a}x <_{\X} y_{a} <_{\X} y_{a}xy, \quad \text{and} \quad (y_{a}xy)_{L} = y_{a}xy.
\end{align*} 
This behavior contrasts with Example~\ref{eg1:xR}, as $y$ is not the greatest letter of $y_{a}xy$ with respect to $<_{X}$.
\end{example}

For any word $\alpha \in \X$, let $m_{\alpha}$ denote the greatest letter occurring in $\alpha$ with respect to $<_{X}$. Similarly, for a polynomial $p = \sum_{i} k_{i}\alpha_{i} \in \K\X$ with $\alpha_{i} \in \X$, we denote by $m_{p}$ the greatest letter across all support words $\alpha_{i}$.

As illustrated by the preceding examples, the mr-factorization of a word depends closely on the original element of its greatest letter. In general, we establish the following properties:

\begin{lemma}\label{lem:properties-alpha_R}
Let $\alpha$ be a non-empty word in $\X$, and let $\mathrm{mrf}(\alpha)=(\alpha_{L},\alpha_{R})$. Set $m:=C(m_{\alpha})$. Then the following assertions hold:
\begin{itemize}
    \item [(a)] $C(C(\alpha)) = C(\alpha)$, $C(r(\alpha)) = r(C(\alpha))$, and $C(m_{\alpha}) = m_{C(\alpha)}$.
    \item [(b)] $C(\alpha_{R})(m) = 0$ and $r(C(\alpha_{L})) = m$. 
    \item [(c)] $r(C(\alpha)) = m$ if and only if $\alpha_{L} = \alpha$.
    \item [(d)] A decomposition $\alpha = \alpha_{1}\alpha_{2}$ constitutes the mr-factorization of $\alpha$ if and only if $r(C(\alpha_{1})) = m$ and $C(\alpha_{2})(m) = 0$.
    \item [(e)] The prefix $\alpha_{L}$ is its own greatest factor with respect to the reduction order $<_{\X}$. Consequently, $\alpha_{L}$ is the greatest factor of $\alpha$ with respect to $<_{\X}$.
\end{itemize}
\end{lemma}

\begin{lemma}\label{lem:C(a_L)=C(a)_L}
Let $\alpha, \beta$ be non-empty words in $\X$. Then the following assertions hold:
\begin{itemize}
\item [(a)] $C(\alpha_{L}) = C(\alpha)_{L}$, which implies $\mathrm{mrf}(C(\alpha)) = (C(\alpha_{L}), C(\alpha_{R}))$.
\item [(b)] If $C(\alpha) = C(\beta)$, then $C(\alpha_{L}) = C(\beta_{L})$.
\item [(c)] The prefix $C(\alpha_{L})$ is its own greatest factor with respect to $<_{\lOr}$. Moreover, $C(\alpha_{L})$ is the greatest factor of $C(\alpha)$ with respect to $<_{\lOr}$.
\end{itemize}
\end{lemma}

By applying Lemma~\ref{lem:properties-alpha_R}, we deduce the following structural property for prime words:

\begin{proposition}\label{prop:prime-words}
Let $w_{1}, w_{2} \in \X_{p}$ such that $w_{1} \le_{\X} w_{2}$. Then their product $w = w_{1}w_{2}$ belongs to $\X_{p}$ and satisfies $w_{1} \le_{\X} w_{2} <_{\X} w$. 
\end{proposition}

Consequently, every non-empty word admits a unique canonical factorization into prime words:

\begin{proposition}\label{prop:unique-factorization}
Every non-empty word $\omega \in \X$ can be uniquely decomposed into a strictly decreasing product of prime words with respect to the reduction order $<_{\X}$; that is,
\begin{align*}
\omega = \omega_{1}\omega_{2}\cdots \omega_{n},
\end{align*} 
where $\omega_{1},\ldots,\omega_{n} \in \X_{p}$, and $\omega_{1} >_{\X} \ldots >_{\X} \omega_{n}$. 
\end{proposition}

By virtue of Proposition~\ref{prop:prime-words}, the unique factorization presented in Proposition~\ref{prop:unique-factorization} achieves the minimal length among all possible decompositions of a non-empty word into a product of prime words.

\section{Irreducible Words and Letters}\label{sec:3}

In this section, we investigate the properties of irreducible words with respect to the reduction order $<_{\X}$. Throughout this section, we fix a well-ordered set $(O, <_{O})$ equipped with a map $t_{O}: O \rightarrow \mathbb{N}$. Under this setting, the set $X = O_{0} \cup \left(\bigcup_{y \in O_{+}} C_{y}\right)$ is well-ordered with respect to $<_{X}$, and consequently, the free monoid $(\X, <_{\X})$ is well-ordered.

The \textit{leading word} of a non-zero polynomial $f \in \K\X$, denoted by $\LW(f)$, is defined as the greatest word appearing in the support of $f$ with respect to the reduction order $<_{\X}$. 

Let $I$ be an ideal of $\K\X$. A word $\alpha \in \X$ is said to be \textit{$I$-reducible} if $\LW(f) = \alpha$ for some polynomial $f \in I$. Conversely, a word in $\X$ is called \textit{$I$-irreducible} if it is not $I$-reducible. 

We denote the corresponding sets of irreducible elements as follows:
\begin{align*}
X_{I} &:= \{x \in X \mid x \text{ is $I$-irreducible} \}, \quad \qquad O_{I} := X_{I} \cap O, \\
X_{I,+} &:= \{x \in X_{+} \mid x \text{ is $I$-irreducible} \}, \qquad O_{I,+} := X_{I,+} \cap O, \\
\X_{I} &:= \{\omega \in \X \mid \omega \text{ is $I$-irreducible} \}.
\end{align*}

\begin{example}
Let $I$ be a proper ideal of $\K\X$. Assume that $a \in O_{0}$, $x \in O_{+}$, and $a <_{O} x$. Suppose further that there exist elements $a_{1}, a_{2} \in O_{0}^{*}$ and scalars $k_{a_1,a_2} \in \K$ such that  
\begin{align*}
ax \in \sum k_{a_1,a_2} x_{a_{1}} a_{2} + I.  
\end{align*}
By applying the order properties, we observe that
\begin{align*}
C(x_{a_{1}} a_{2}) \le_{\lOr} C((x_{a_{1}} a_{2})_{L}) = C(x_{a_{1}}) = C(x) <_{\lOr} C(ax).  
\end{align*}
This implies that $x_{a_{1}} a_{2} <_{\X} ax$, from which it follows that $\LW(f) = ax$ for some $f \in I$. Hence, the word $ax$ is $I$-reducible.
\end{example}

It is easily verified via a standard reduction argument that any word can be expressed modulo $I$ as a linear combination of irreducible words. This leads to the following foundational basis property:

\begin{proposition}\label{prop:irrws-form-basis}
Let $I$ be an ideal of $\K\X$. The residue classes of the $I$-irreducible words, given by $\{\omega + I \mid \omega \in \X_{I} \}$, form a $\K$-basis of the quotient algebra $\K\X/ I$.
\end{proposition}

Let $\K\X$ be the augmented free algebra equipped with the canonical augmentation map $\epsilon: \K\X \rightarrow \K$, where $\epsilon(X) = 0$. We obtain the following structural consequence regarding subwords:

\begin{proposition}\label{prop:prefix-irrws}
Let $(\K\X, \epsilon)$ be the augmented free algebra on $X$, and let $I$ be a proper ideal of $\K\X$ such that $I \subseteq \ker \epsilon$. Then every suffix of an $I$-irreducible word in $\X$ is also $I$-irreducible.
\end{proposition}

\begin{remark}
We note that the property of irreducibility is asymmetric with respect to word boundaries; specifically, a prefix of an $I$-irreducible word is not necessarily $I$-irreducible.
\end{remark}

For any letter $y \in X$, we introduce the following bounded subsets of letters:
\begin{align*}
X^{< y} := \{x \in X \mid x <_{X} y\}, \qquad X_{I}^{< y} := \{x \in X_{I} \mid x <_{X} y \}.
\end{align*}
Using these sets, we establish a linear independence property for irreducible letters in $O$.

\begin{lemma}\label{lem:irrl}
Let $(\K\X, \epsilon)$ be the augmented free algebra on $X$, and let $I \subseteq \ker \epsilon$ be a proper ideal of $\K\X$. If $y$ is an $I$-irreducible letter in $O$, then 
\begin{align*}
y \notin \K\X \cdot X^{< y} + I.
\end{align*}
\end{lemma}

\begin{proof}
Suppose for contradiction that the assertion is false. Then there exists an $I$-irreducible letter $y \in O$ such that 
\begin{align*}
y \in \K\X \cdot X^{< y} + I.
\end{align*}
Consequently, there exist polynomials $f_{i} \in \K\X$, letters $x_{i} \in X^{< y}$, and scalars $k_{i} \in \K$ such that the polynomial
\begin{align*}
p = y - \sum_{i} k_{i}f_{i}x_{i}
\end{align*}
belongs to $I$. Since the rightmost letter of each term satisfies $r(f_{i}x_{i}) = x_{i} <_{X} y$, Corollary~\ref{cor:well-orders-compatibility-2} ensures that $C(f_{i}x_{i}) <_{\lOr} y = C(y)$ for all indices $i$. From the definition of the reduction order on $\X$, it follows that $f_{i}x_{i} <_{\X} y$, which yields $\LW(p) = y$. Since $p \in I$, this implies that $y$ is $I$-reducible, yielding the desired contradiction.
\end{proof}

Given an augmented algebra $(A, \epsilon_{A})$ over $\K$, we select a generating set $X$ of $A$ such that $X \subseteq \ker \epsilon_{A}$. By a slight abuse of notation, we also let $X$ denote the generating set of the free algebra $\K\X$. Let $\pi: \K\X \rightarrow A$ be the canonical projection lifted from the assignment $f_{X}: X \rightarrow A$. Under this construction, $(\K\X, \epsilon)$ becomes an augmented algebra whose augmentation map $\epsilon: \K\X \rightarrow \K$ is given by the composition $\epsilon = \epsilon_{A} \circ \pi$. Let $I := \ker \pi$. It is immediate that $X, I \subseteq \ker \epsilon$.

Assume that $X = O_{0} \cup \left(\bigcup_{y \in O_{+}} C_{y}\right)$. We demonstrate that the corresponding set of irreducible letters is finite whenever the augmented algebra $A$ satisfies the left Noetherian property.

\begin{lemma}\label{lem:O_I-finite}
Let $(A, \epsilon_{A})$ be an augmented algebra generated by the set $X = O_{0} \cup \left(\bigcup_{y \in O_{+}} C_{y}\right)$. Let $\pi: \K\X \rightarrow A$ be the canonical projection and let $I := \ker \pi$. If $A$ is left Noetherian, then the sets $O_{I}$ and $O_{I,+}$ are finite. 
\end{lemma}

\begin{proof}
Suppose for contradiction that $O_{I}$ is infinite. We can therefore choose an infinite sequence of distinct letters $x_{1}, x_{2}, \ldots, x_{n}, \ldots$ in $O_{I}$ ordered such that 
$$x_{1} <_{O} x_{2} <_{O} \ldots <_{O} x_{n} <_{O} \ldots.$$
Using these elements, we construct the following ascending chain of left ideals in $\K\X$:
\begin{align*}
\K\X \cdot x_{1} + I \subseteq \K\X \cdot \{x_{1},x_{2}\} + I \subseteq \ldots \subseteq \K\X \cdot \{x_{1},\ldots,x_{n}\} + I \subseteq \ldots.
\end{align*}
Since $X, I \subseteq \ker \epsilon$, it follows that $1 \notin \K\X \cdot \{x_{1},\ldots,x_{n}\} + I$, ensuring that each left ideal in the chain is proper. Furthermore, Lemma~\ref{lem:irrl} implies that 
\begin{align*}
x_{n+1} \notin \K\X \cdot \{x_{1},\ldots,x_{n}\} + I \qquad \text{for all } n \ge 1.
\end{align*} 
Consequently, the chain of left ideals is strictly ascending.

Since $\pi$ is an algebra homomorphism, applying $\pi$ to this sequence yields a strictly ascending chain of left ideals in $A$:
\begin{align*}
A \cdot \pi(x_{1}) \subsetneq A \cdot \{\pi(x_{1}),\pi(x_{2})\} \subsetneq \ldots \subsetneq A \cdot \{\pi(x_{1}),\ldots,\pi(x_{n})\} \subsetneq \ldots.
\end{align*}
This directly contradicts the hypothesis that $A$ is left Noetherian. Therefore, $O_{I}$ must be finite, which immediately implies that its subset $O_{I,+}$ is finite.
\end{proof}

\section{Comultiplication on a Free Algebra}\label{sec:4}

This section utilizes the mirror reduction order, reduction-factorization, and prime words introduced in Section~\ref{sec:2} to determine the action of comultiplication on arbitrary words.

For a set $X$ equipped with a map $t_{X}: X \rightarrow \mathbb{N}$, the free algebra $\K\X$ on $X$ admits a grading $\K\X = \bigoplus_{n\ge 0} \K\X_{n}$ induced by setting $\deg(x) = t_{X}(x)$ for each $x \in X$. 

Assume that $X = O_{0} \cup \left(\bigcup_{y\in O_{+}} C_{y}\right)$ is the well-ordered set defined in Section~\ref{sec:2}, equipped with the map $t_{X}$ and the well order $<_{X}$. We introduce an algebraic comultiplication structure on the free algebra $\K\X$ as follows.

\begin{definition}\label{def:comultiplication}
A \textit{comultiplication} on $\K\X$ is an algebra homomorphism $\Delta: \K\X \rightarrow \K\X \otimes \K\X$. A comultiplication $\Delta$ on $\K\X$ is said to be \textit{skew-triangular} if it satisfies the following two conditions:
\begin{itemize}
\item For each $x \in X_{0}$, there exist elements $x' \in X_{0}$ and $r_{x}, x'' \in \K X_{0}$ such that
\begin{align*}
\Delta(x) = x \otimes (1 - r_{x}) + 1 \otimes x + \sum x' \otimes x'', \quad \text{with } x' \neq x;
\end{align*}
\item For each $x \in X_{+}$, there exist elements $x'_{+} \in X_{+}^{*}$, $x'_{0} \in X_{0}^{*}$, and $x'' \in \K\X^{+}$ such that
\begin{align*}
\Delta(x) = x \otimes 1 + 1 \otimes x + \sum x'_{+}x'_{0} \otimes x'',
\end{align*}
where $x'_{+}x'_{0} \neq 1$, $t_{X}(x'_{+}) < t_{X}(x)$, and $C(x'_{+}) <_{O} C(x)$.
\end{itemize}
\end{definition}

For convenience, we set $r_{x} := 0$ for each $x \in X_{+}$. 
Given a word $\omega = x_{1}\cdots x_{n} \in \X$ with $x_{1},\ldots,x_{n} \in X$, we adopt the shorthand notation:
\begin{align*}
[1-r_{\omega}] &:= \prod_{i=1}^{n}(1-r_{x_{i}}) \in \K\Xo.
\end{align*}

For a Hopf algebra whose coradical forms a Hopf subalgebra, we demonstrate that its coproduct can be lifted to a skew-triangular comultiplication on an appropriately constructed free algebra $\K\X$. 

\begin{lemma}\label{lem:X-generates-H}
Let $H$ be a Hopf algebra over $\K$ satisfying $S_{H}(H_{[0]}) \subseteq H_{[0]}$. Then there exist well-ordered generating sets $(O, <_O)$ and $(X, <_{X})$ of $H$, accompanied by assignments $f_{O}: O \rightarrow H$, $f_{X}: X \rightarrow H$ and degree maps $t_{O}: O \rightarrow \mathbb{N}$, $t_{X}: X \rightarrow \mathbb{N}$, such that the following assertions hold:
\begin{itemize}
\item [(a)] $X = O_{0} \cup \left(\bigcup_{y \in O_{+}} C_{y}\right)$, $C_{y} = \{y_{a} \mid a \in O_{0}^{*}\}$, and $X_{0} = O_{0}$;
\item [(b)] $f_{X}|_O = f_{O}$, $f_{X}(X) \subseteq \ker \epsilon_{H}$, $\K f_{X}(X_{0}) = \ker \epsilon_{H_{[0]}}$, and for any $a\in O_{0}^{*}$ and $y \in O_{+}$,
\begin{align*}
f_{X}(y_{a}) = \sum f_{O}(a)_{(1)} f_{O}(y) S_{H}(f_{O}(a)_{(2)});
\end{align*} 
\item [(c)] $\pi: \K\X \rightarrow H$ is the canonical projection lifted from $f_{X}$;
\item [(d)] $\K\X$ admits a skew-triangular comultiplication $\Delta$ and an augmentation $\epsilon$ satisfying
\begin{align*}
\Delta(I) \subseteq I \otimes \K\X + \K\X \otimes I, \quad \text{and} \quad X, I \subseteq \ker \epsilon,
\end{align*}
where $I := \ker \pi$.
\end{itemize}
\end{lemma}

\begin{proof}
By Lemma~\ref{lem:gr[c]H-Radford-biproduct}, we have an isomorphism of graded Hopf algebras $\gr_{\s} H \cong B \sharp H_{[0]}$, where $B = \bigoplus_{n\ge 0} B_{n}$ is a graded connected Hopf algebra in the Yetter-Drinfel'd category ${^{H_{[0]}}_{H_{[0]}}\mathcal{YD}}$, and $(\gr_{\s} H)_{n} = B_{n}H_{[0]}$. 

For each $b \in B_{n}$ with $n \ge 1$, there exist elements $b_{1} \in B_{n-k}$ and $b_{2} \in B_{k}$ ($1 \le k \le n$) such that
\begin{align*}
\Delta_{B}(b) = b \otimes 1 + 1 \otimes b + \sum b_{1} \otimes b_{2}.
\end{align*}
Consequently, there exist elements $b_{(-1)}, b_{2(-1)} \in H_{[0]}$, $b_{(0)} \in B_{n}$, and $b_{2(0)} \in B_{k}$ ($1 \le k \le n$) such that
\begin{align*}
\Delta_{\gr_{\s} H}(b) = b \otimes 1 + \sum b_{(-1)} \otimes b_{(0)} + \sum \sum b_{1}b_{2(-1)} \otimes b_{2(0)}.
\end{align*}
Applying the counit axiom for $H_{[0]}$-comodules, we can express $\rho_{B}(b) = (1-z_{b}) \otimes b + v$ for some elements $z_{b} \in \ker \epsilon_{H_{[0]}}$ and $v \in \ker \epsilon_{H_{[0]}} \otimes (\gr_{\s} H)_{n}$. Hence, we can write 
\begin{align*}
\Delta_{\gr_{\s} H}(b) = b \otimes 1 + 1 \otimes b + \sum b' \otimes b''
\end{align*}
for some $b' \in (\gr_{\s} H)_{n-k} = B_{n-k}H_{[0]}$ and $b'' \in (\gr_{\s} H)_{k} = B_{k}H_{[0]}$ ($1 \le k \le n$).

Next, define the homogeneous filtration components:
\begin{align*}
H_{B,n} &:= \{ d \in H_{[n]} \mid 0 \neq d + H_{[n-1]} \in B_{n} \}, \quad n \ge 1, \qquad H_{B} := \bigcup_{n \ge 1} H_{B,n}.
\end{align*} 
It is straightforward to verify that 
\begin{align*}
H_{[n]} = \K H_{B,n}H_{[0]} + H_{[n-1]}, \qquad H = \K H_{B}H_{[0]} + H_{[0]}.
\end{align*}
Thus, for each $d \in H_{B,n} \cap \ker \epsilon_{H}$, there exist elements $d'_{+} \in \bigcup_{k=1}^{n-1} H_{B,k} \cup \{1\}$, $d'_{0} \in H_{[0]}$, and $d'' \in H_{[n]} \cap \ker\epsilon_{H}$ such that 
\begin{align*}
\Delta_{H}(d) = d \otimes 1 + 1 \otimes d + \sum d'_{+} d'_{0} \otimes d'', \quad \text{with } d'_{+}d'_{0} \in \ker \epsilon_{H}.
\end{align*}

Let $L$ be a basis of the augmentation ideal $\ker \epsilon_{H_{[0]}}$ of $H_{[0]}$. 
For each $z \in L$, since $\Delta_{H}(z) - z \otimes 1 - 1 \otimes z \in \K L \otimes \K L$, there exist elements $z' \in L$ and $r_{z}, z'' \in \K L $ such that
\begin{align*}
\Delta_{H}(z) = z \otimes (1-r_{z}) + 1 \otimes z + \sum z' \otimes z'', \quad \text{with } z' \neq z.
\end{align*}

We now choose a set $O$ of formal variables to serve as a generating set for $H$, equipped with an assignment map $f_{O}: O \rightarrow H$ and a degree map $t_{O}: O \rightarrow \mathbb{N}$ defined by $$t_{O}(y) = \min\{n \mid f_{O}(y) \in H_{[n]}\}.$$ Here, $O_{0}$ denotes the set of formal variables corresponding to $L$, chosen such that
\begin{align*}
f_{O}(O) \subseteq \ker \epsilon_{H} \quad \text{and} \quad f_{O}(O_{0}) = L.
\end{align*}
We define the expanded variable sets:
\begin{align*}
X &:= O_{0} \cup \left(\bigcup_{y\in O_{+}} C_{y}\right), \qquad C_{y} := \{y_{a} \mid a \in O_{0}^{*}\} \quad \text{for } y \in O_{+},
\end{align*}
and define the assignment $f_{X}: X \rightarrow H$ via
\begin{align*}
f_{X}(x) &= f_{O}(x), \qquad \text{for } x \in O,\\
f_{X}(y_{a}) &= \sum f_{O}(a)_{(1)} f_{O}(y) S_{H}(f_{O}(a)_{(2)}), \qquad \text{for } y_{a} \in C_{y}, ~y \in O_{+}, ~ a \in O_{0}.
\end{align*}
Evidently, $f_{X}(X) \subseteq \ker \epsilon_{H}$. Let $\pi: \K\X \rightarrow H$ be the canonical algebra projection lifted from $f_{X}$, and let $I := \ker \pi$. Recall that the induced map $t_{X}: X \rightarrow \mathbb{N}$ satisfies $t_{X}(a) = t_{O}(a)$ and $t_{X}(y_{b}) = t_{X}(y) = t_{O}(y)$ for $a,b \in O_{0}$ and $y \in O_{+}$. It follows that $X_{0} = O_{0}$, and the free algebra $\K\X$ inherits the grading $\K\X = \bigoplus_{n\ge 0} \K\X_{n}$ determined by $\deg(x) = t_{X}(x)$ for $x \in X$.

By the commutation property of the left adjoint action and the local finiteness of coalgebras, the generating sets $O$ and $X$ can be chosen such that $\Delta_{H}$ lifts to an algebra homomorphism $\Delta: \K\X \rightarrow \K\X \otimes \K\X$ satisfying $\Delta_{H} \circ \pi = (\pi \otimes \pi) \circ \Delta$, subject to the following conditions:
\begin{itemize}
\item For every $x \in X_{0}$, there exist elements $x' \in X_{0}$ and $r_{x}, x'' \in \K X_{0}$ such that
\begin{align*}
\Delta(x) = x \otimes (1-r_{x}) + 1 \otimes x + \sum x' \otimes x'', \quad \text{with } x' \neq x;
\end{align*}
\item For every $x \in X_{+}$, there exist elements $x'_{+} \in X_{+}^{*}$, $x'_{0} \in X_{0}^{*}$, and $x'' \in \K\X^{+}$ such that
\begin{align*}
\Delta(x) = x \otimes 1 + 1 \otimes x + \sum x'_{+}x'_{0} \otimes x'', \quad \text{with } x'_{+}x'_{0} \neq 1 \text{ and } \ t_{X}(x'_{+}) < t_{X}(x).
\end{align*}
\end{itemize}
By construction, $\Delta$ defines a comultiplication on $\K\X$ such that $\Delta(I) \subseteq I \otimes \K\X + \K\X \otimes I$. Analogously, the counit $\epsilon_{H}$ lifts to an algebra map $\epsilon: \K\X \rightarrow \K$ satisfying $X, I \subseteq \ker \epsilon$.

Finally, we impose a total order $<_{X}$ on $X$. We first fix a well-ordering $<_{O_{n}}$ on each homogeneous component $O_{n}$ for $n \ge 0$. For any $x,y \in O$, we define:
\begin{align*}
x <_{O} y \Longleftrightarrow 
\left\{\begin{array}{l}
t_{O}(x) < t_{O}(y), \text{ or } \\
t_{O}(x) = t_{O}(y) = n \text{ and } x <_{O_{n}} y.
\end{array}\right.
\end{align*}
The induced total order $<_{X}$ on $X$ is then given as follows: for any $x, y \in X$, 
\begin{align*}
x <_{X} y \Longleftrightarrow 
\left\{\begin{array}{l}
t_{X}(x) < t_{X}(y), \text{ or } \\
t_{X}(x) = t_{X}(y) = n \text{ and } 
\left\{\begin{array}{l}
C(x) <_{O_{n}} C(y), \text{ or }\\
C(x) = C(y) = z \in O_{+} \text{ and } x <_{C_{z}} y.
\end{array}\right.
\end{array}\right.
\end{align*}
It is straightforward to verify that $(O, <_{O})$ forms a well-ordered set. By Lemma~\ref{lem:well-orders}, the sets $(X, <_{X})$, $(\lOr, <_{\lOr})$, and $(\X, <_{\X})$ are well-ordered. Moreover, the lifted comultiplication $\Delta$ is skew-triangular by construction.
\end{proof}

\begin{remark}\label{rmk:ab-c-I}
By the construction of $X_{0}$ (or equivalently, $O_{0}$), it is clear that for any $a, b \in X_{0}$, there exist elements $c \in X_{0}$ and scalars $k_{c} \in \K$ such that
$$ab \in \sum k_{c} c + I.$$
\end{remark}

Throughout the remainder of this section, we adopt the framework established in Lemma~\ref{lem:X-generates-H}, preserving all conventions and notations introduced in its proof. We now establish the following result regarding the structure of products:

\begin{lemma}\label{lem:zw-wz-I}
Let $a \in X_{0}$ and let $\omega = x_{1} \cdots x_{n} \in \Xp$ with $x_{i} \in X_{+}$. Set $y_{i} := C(x_{i})$ for each $1 \le i \le n$. Then there exist elements $a_{1}, \ldots, a_{n+1} \in X_{0}^{*}$, letters $(y_{1})_{a_{1}} \in C_{y_{1}}, \ldots, (y_{n})_{a_{n}} \in C_{y_{n}}$, and coefficients $k_{a_{1}, \ldots, a_{n+1}} \in \K$ such that
\begin{align*}
a \omega \in \sum k_{a_{1},\ldots,a_{n+1}} (y_{1})_{a_{1}} \cdots (y_{n})_{a_{n}} a_{n+1} + I, 
\end{align*}
where $(y_{1})_{a_{1}} \cdots (y_{n})_{a_{n}} a_{n+1} \le_{\X} ((y_{1})_{a_{1}} \cdots (y_{n})_{a_{n}} a_{n+1})_{L} <_{\X} a\omega$. More precisely,
\begin{align*}
C(((y_{1})_{a_{1}} \cdots (y_{n})_{a_{n}} a_{n+1})_{L}) = C((y_{1})_{a_{1}} \cdots (y_{n})_{a_{n}}) = C(\omega) <_{\lOr} C(a\omega).
\end{align*}
Consequently, the word $a\omega$ is $I$-reducible.
\end{lemma}

\begin{proof}
By induction on the word length, it suffices to consider the base case where $\omega = y_{b} \in X_{+}$ for some $y \in O_{+}$ and $b \in O_{0}^{*}$. By utilizing the canonical algebra projection $\pi$ constructed in Lemma~\ref{lem:X-generates-H}, we obtain:
\begin{align*}
\pi(ay_{b}) &= \pi(a) \pi(y_{b}) \\
& = \sum \pi(a) \pi(b)_{(1)} \pi(y) S_{H}(\pi(b)_{(2)}) \\
& = \sum \pi(a)_{(1)} \pi(b)_{(1)} \pi(y) S_{H}( \pi(b)_{(2)} ) S_{H}( \pi(a)_{(2)(1)} ) \pi(a)_{(2)(2)} \\
& = \sum \pi(a)_{(1)(1)} \pi(b)_{(1)} \pi(y) S_{H}( \pi(b)_{(2)} ) S_{H}( \pi(a)_{(1)(2)} ) \pi(a)_{(2)}.
\end{align*}
By virtue of the comultiplication $\Delta$ defined in Lemma~\ref{lem:X-generates-H}, there exist elements $a', a'' \in X_{0}^{*}$ and scalars $k_{a',a''} \in \K$ such that
\begin{align*}
\Delta_{H}(\pi(a)) &= (\pi \otimes \pi)(\Delta(a)) = \sum k_{a',a''} \pi(a') \otimes \pi(a'').
\end{align*}
Substituting this expression into the identity for $\pi(ay_{b})$ yields:
\begin{align*}
\pi(ay_{b}) & = \sum k_{a',a''} \pi(a')_{(1)} \pi(b)_{(1)} \pi(y) S_{H}(\pi(b)_{(2)}) S_{H}(\pi(a')_{(2)}) \pi(a'') \\
& = \sum k_{a',a''} \pi(a'b)_{(1)} \pi(y) S_{H}(\pi( a'b)_{(2)} ) \pi(a'').
\end{align*} 
From Remark~\ref{rmk:ab-c-I}, it follows that each product $a' b$ can be  expressed modulo $I$ as $a' b \in \sum k_{a_{1}} a_{1} + I$ for some $a_{1} \in X_{0}^{*}$ and $k_{a_{1}} \in \K$. Setting $a_{2} := a''$ and $k_{a_{1},a_{2}} := k_{a_{1}} k_{a',a''}$, we obtain:
\begin{align*}
\pi(ay_{b}) = \sum \sum k_{a_{1},a_{2}} \pi(a_{1})_{(1)} \pi(y) S_{H}(\pi( a_{1})_{(2)} ) \pi(a_{2}) = \sum \sum k_{a_{1},a_{2}} \pi(y_{a_{1}}) \pi(a_{2}).
\end{align*}
This equality implies that $ay_{b} \in \sum k_{a_{1},a_{2}} y_{a_{1}} a_{2} + I$. Furthermore, we observe that $C(y_{a_{1}}a_{2}) <_{\lOr} C(ay_{b})$ holds because
$$C(y_{a_{1}}a_{2}) \le_{\lOr} C( (y_{a_{1}}a_{2})_{L} ) = C(y_{a_{1}}) = y <_{\lOr} ay = C(ay_{b}).$$
By definition of the reduction order on $\X$, this inequality ensures that $y_{a_{1}}a_{2} <_{\X} ay_{b}$. Consequently, the leading word of the corresponding relation in $I$ is $ay_{b}$, proving that $ay_{b}$ is $I$-reducible. 
\end{proof}

Next, we determine the action of comultiplication on the prime words in $\Xp$.

\begin{lemma}\label{lem:comultiplication-prime-words}
Let $\omega$ be a prime word in $\Xp$. Then there exist non-empty words $\omega' \in \Xp\Xo$ and polynomials $\omega'' \in \K\X^{+}$ such that
\begin{align*}
\Delta(\omega) \in \omega \otimes 1 + 1 \otimes \omega + \sum \omega' \otimes \omega'' + I \otimes \K\X,
\end{align*}
where $C(\omega') \le_{\lOr} C((\omega')_{L}) <_{\lOr} C(\omega)$, which implies $\omega' \le_{\X} (\omega')_{L} <_{\X} \omega$.
\end{lemma}

\begin{proof}
We proceed by induction on the length of the word $\omega$. For $|\omega| = 1$, the assertion follows immediately from the definition of the skew-triangular comultiplication on $X$. Now, assume $|\omega| \ge 2$ and decompose the word as $\omega = x\alpha$ for some letter $x \in X$. Since $r(C(\alpha)) = r(C(\omega)) = C(m_{\omega}) = C(m_{\alpha})$, Lemma~\ref{lem:properties-alpha_R}~(c) implies that $\alpha_{L} = \alpha$. 

Consequently, by the induction hypothesis (since $|\alpha| < |\omega|$), there exist elements $x'_{+} \in X_{+}^{*}$, $x'_{0} \in X_{0}^{*}$, $\alpha'_{+} \in \Xp$, and $\alpha'_{0} \in \Xo$ (setting $\alpha' := \alpha'_{+}\alpha'_{0}$), along with polynomials $x'', \alpha'' \in \K\X^{+}$, such that
\begin{align*}
\Delta(\omega) =& \ \Delta(x)\Delta(\alpha) \\
=&  \left(x \otimes 1 + 1 \otimes x + \sum x'_{+}x'_{0} \otimes x''\right)\left(\alpha \otimes 1 + 1 \otimes \alpha + \sum \alpha' \otimes \alpha''\right)\\
=& \  x\alpha \otimes 1 + x \otimes \alpha + \sum x\alpha' \otimes \alpha''\\
& + \alpha \otimes x + 1 \otimes x\alpha + \sum \alpha' \otimes x\alpha''\\
& + \sum x'_{+}x'_{0} \alpha \otimes x'' + \sum x'_{+}x'_{0} \otimes x'' \alpha + \sum \sum x'_{+}x'_{0} \alpha' \otimes x''\alpha'',
\end{align*} 
where $x'_{+}x'_{0} \neq 1$, $\alpha' \neq 1$, $C(x'_{+}) <_{O} C(x)$, and $C(\alpha') \le_{\lOr} C((\alpha')_{L}) = C((\alpha'_{+})_{L}) <_{\lOr} C(\alpha)$. 

By Lemma~\ref{lem:zw-wz-I}, there exist elements $x'_{01}, x'_{02} \in X_{0}^{*}$, scalars $k_{x_{01}',x'_{02}} \in \K$, and non-empty prime words $\alpha_{x'_{01}} \in \Xp$ such that 
\begin{align*}
x_{0}' \alpha \in \sum k_{x_{01}',x'_{02}} \alpha_{x'_{01}} x'_{02} + I, \qquad \text{with } C(\alpha_{x'_{01}}) = C(\alpha).
\end{align*}
Similarly, for each term $\alpha'$ satisfying $\alpha'_{+} \neq 1$, there exist scalars $\ell_{x_{01}',x'_{02}} \in \K$ and non-empty prime words $(\alpha'_{+})_{x'_{01}} \in \Xp$ such that
\begin{align*}
x_{0}' \alpha'_{+} \in \sum \ell_{x_{01}',x'_{02}} (\alpha'_{+})_{x'_{01}} x'_{02} + I, \qquad \text{with } C((\alpha'_{+})_{x'_{01}}) = C(\alpha'_{+}).
\end{align*}
Substituting these relations back into the expression for $\Delta(\omega)$, we obtain the following containment modulo $I \otimes \K\X$:
\begin{align*}
\Delta(\omega) \in  
& \ x \alpha \otimes 1 + x \otimes \alpha + \sum x\alpha' \otimes \alpha'' + \alpha \otimes x + 1 \otimes x\alpha + \sum \alpha' \otimes x\alpha''\\
& + \sum \sum k_{x'_{01},x'_{02}} x'_{+} \alpha_{x'_{01}} x'_{02} \otimes x'' + \sum x'_{+}x'_{0} \otimes x'' \alpha + \sum \sum_{\alpha'_{+} = 1} x'_{+} x'_{0} \alpha'_{0} \otimes x''\alpha'' \\
& + \sum \sum_{\alpha'_{+} \neq 1} \sum \ell_{x'_{01},x'_{02}} x'_{+}(\alpha'_{+})_{x'_{01}} x'_{02} \alpha'_{0} \otimes x''\alpha'' + I \otimes \K\X.
\end{align*}

By virtue of Lemma~\ref{lem:C(a_L)=C(a)_L}, it suffices to verify that every left-tensor word $\omega'$ appearing in the sums above satisfies the condition $C((\omega')_{L}) <_{\lOr} C(\omega)$. Since $\omega$ is assumed to be a prime word, Lemma~\ref{lem:C(a_L)=C(a)_L} ensures that $C(\alpha) <_{\lOr} C(\omega)$ and $C(x) <_{\lOr} C(\omega)$. We now check this condition case by case using Lemmas~\ref{lem:ac<bc} and \ref{lem:C(a_L)=C(a)_L}:

\begin{itemize}
\item \textit{Case $\omega' = x$}: 
We have $C(x_{L}) = C(x) <_{\lOr} C(\omega)$.

\item \textit{Case $\omega' = x\alpha'$}:
Here, $C((x\alpha')_{L}) = (C(x)C(\alpha'))_{L}$. 
\begin{itemize}
    \item If $m_{C(x)} >_{O} m_{C(\alpha')}$, then $(C(x)C(\alpha'))_{L} = C(x) <_{\lOr} C(\omega)$.
    \item If $m_{C(x)} \le_{O} m_{C(\alpha')}$, then $(C(x)C(\alpha'))_{L} = C(x) C(\alpha')_{L} <_{\lOr} C(x) C(\alpha) = C(\omega)$.
\end{itemize}

\item \textit{Case $\omega' = \alpha$}:  
We have $C(\alpha_{L}) = C(\alpha) <_{\lOr} C(\omega)$.

\item \textit{Case $\omega' = \alpha'$}:  
We have $C((\alpha')_{L}) <_{\lOr} C(\alpha) <_{\lOr} C(\omega)$. 

\item \textit{Case $\omega' = x'_{+} \alpha_{x'_{01}} x'_{02}$}:  
Here, $C((x'_{+} \alpha_{x'_{01}} x'_{02})_{L}) = (C(x'_{+})C(\alpha))_{L}$.
\begin{itemize}
    \item If $m_{C(x'_{+})} >_{O} m_{C(\alpha)}$, then $(C(x'_{+})C(\alpha))_{L} = C(x'_{+}) <_{\lOr} C(x) <_{\lOr} C(\omega)$.
    \item If $m_{C(x'_{+})} \le_{O} m_{C(\alpha)}$, then $(C(x'_{+})C(\alpha))_{L} = C(x'_{+}) C(\alpha)_{L} = C(x'_{+}) C(\alpha) <_{\lOr} C(x) C(\alpha) = C(\omega)$.
\end{itemize}

\item \textit{Case $\omega' = x'_{+}x'_{0}$}:  
We have $C((x'_{+}x'_{0})_{L}) = C(x'_{+}) <_{\lOr} C(x) <_{\lOr} C(\omega)$. 

\item \textit{Case $\omega' = x'_{+}x'_{0} \alpha'_{0}$}: 
We have $C((x'_{+}x'_{0} \alpha'_{0})_{L}) = C(x'_{+}) <_{\lOr} C(x) <_{\lOr} C(\omega)$.

\item \textit{Case $\omega' = x'_{+}(\alpha'_{+})_{x'_{01}} x'_{02} \alpha'_{0}$ with $\alpha_{+}' \neq 1$}: 
Here, $C((x'_{+}(\alpha'_{+})_{x'_{01}} x'_{02} \alpha'_{0})_{L}) = (C(x'_{+}) C(\alpha'_{+}))_{L}$.
\begin{itemize}
    \item If $m_{C(x'_{+})} >_{O} m_{C(\alpha'_{+})}$, then $(C(x'_{+})C(\alpha'_{+}))_{L} = C(x'_{+}) <_{\lOr} C(x) <_{\lOr} C(\omega)$.
    \item If $m_{C(x'_{+})} \le_{O} m_{C(\alpha'_{+})}$, then $(C(x'_{+})C(\alpha'_{+}))_{L} = C(x'_{+}) C(\alpha'_{+})_{L} <_{\lOr} C(x'_{+})C(\alpha) <_{\lOr} C(x)C(\alpha) = C(\omega)$.
\end{itemize}
\end{itemize}

Combining all cases, it follows that every resulting term satisfies $C(\omega') \le_{\lOr} C((\omega')_{L}) <_{\lOr} C(\omega)$, and consequently $\omega' \le_{\X} (\omega')_{L} <_{\X} \omega$, completing the proof.
\end{proof}
We now determine the action of comultiplication on arbitrary non-prime words in $\Xp\Xo$.

\begin{lemma}\label{lem:comultiplication-words-2}
Let $\omega$ be a non-empty non-prime word in $\Xp\Xo$, and let $\mathrm{mrf}(\omega) = (\omega_{L}, \omega_{R})$. Then the following assertions hold:
\begin{itemize}
\item [$\mathrm{(a)}$] There exist non-empty words $(\omega_{L})', (\omega_{R})' \in \Xp\Xo$ and polynomials $(\omega_{L})'', (\omega_{R})'' \in \K\X^{+}$ such that 
\begin{align*}
\Delta(\omega_{L}) &\in \omega_{L} \otimes 1 + 1 \otimes \omega_{L} + \sum (\omega_{L})' \otimes (\omega_{L})'' + I \otimes \K\X,\\
\Delta(\omega_{R}) &\in \omega_{R} \otimes [1-r_{\omega_{R}}] + 1 \otimes \omega_{R} + \sum (\omega_{R})' \otimes (\omega_{R})'' + I \otimes \K\X,
\end{align*}
where the greatest letters satisfy $C(m_{(\omega_{L})'}) \le_{O} C(m_{\omega_{L}}) = C(m_{\omega})$ and $C(m_{(\omega_{R})'}) \le_{O} C(m_{\omega_{R}}) < C(m_{\omega})$, and the prime words satisfy $C((\omega_{L})') \le_{\lOr} C(((\omega_{L})')_{L}) <_{\lOr} C(\omega_{L})$.
\item [$\mathrm{(b)}$] There exist non-empty words $\omega' \in \Xp\Xo$ and polynomials $\omega'' \in \K\X^{+}$ such that 
\begin{align*}
\Delta(\omega) &\in \omega \otimes [1-r_{\omega}] + \omega_{L} \otimes \omega_{R} + 1 \otimes \omega + \sum \omega' \otimes \omega'' + I \otimes \K\X,
\end{align*}
where either $C((\omega')_{L}) <_{\lOr} C(\omega_{L})$ holds, or $\omega' = \omega_{L}f$ for some word $f \in \X^{+}$ satisfying $C(m_{f}) <_{O} C(m_{\omega_{L}})$.
\end{itemize}
\end{lemma}

\begin{proof}
Since $\omega_{L}$ is a non-empty prime word in $\Xp$, Part (a) follows directly from Lemmas~\ref{lem:zw-wz-I} and \ref{lem:comultiplication-prime-words}. For Part (b), expanding the product via the algebra homomorphism property yields:
\begin{align*}
\Delta(\omega) = & \  \Delta(\omega_{L}) \Delta(\omega_{R}) \\
\in & \left(\omega_{L} \otimes 1 + 1 \otimes \omega_{L} + \sum (\omega_{L})' \otimes (\omega_{L})''\right)\left(\omega_{R} \otimes [1-r_{\omega_{R}}] + 1 \otimes \omega_{R} + \sum (\omega_{R})' \otimes (\omega_{R})''\right) \\
& + I \otimes \K\X \\
\subseteq & \;\omega_{L}\omega_{R} \otimes [1-r_{\omega_{R}}] + \omega_{L} \otimes \omega_{R} + \sum \omega_{L} (\omega_{R})' \otimes (\omega_{R})'' + I \otimes \K\X \\
& + \omega_{R} \otimes \omega_{L} [1-r_{\omega_{R}}] + 1 \otimes \omega_{L} \omega_{R} + \sum (\omega_{R})' \otimes \omega_{L} (\omega_{R})'' \\
& + \sum (\omega_{L})' \omega_{R} \otimes (\omega_{L})'' [1-r_{\omega_{R}}] + \sum (\omega_{L})' \otimes (\omega_{L})'' \omega_{R} + \sum \sum (\omega_{L})' (\omega_{R})' \otimes (\omega_{L})'' (\omega_{R})''.
\end{align*}

A straightforward application of Part (a) and Lemma~\ref{lem:zw-wz-I} verifies that each left-tensor word $\omega'$ arising in the expansion above satisfies the stated structural constraints.
\end{proof}

\section{Words and Irreducible Letters}\label{sec:5}

In this section, we utilize the comultiplication framework established in Section~\ref{sec:4} to prove that every non-empty word can be expressed as a linear combination of products of degree-zero letters and positive-degree irreducible letters.

Throughout this section, we work within the framework of Lemma~\ref{lem:X-generates-H}, retaining all notations introduced in its proof. We begin by establishing a foundational reduction lemma for prime words. This result plays a central role in proving Lemmas~\ref{lem:beta-u-I}--\ref{lem:rl-sum-smaller-irrls} and Theorem~\ref{thm:rl-sum-lower-irrls}.

\begin{lemma}\label{lem:beta-qbeta-h-I}
Let $\beta$ be a prime word in $\Xp$. Suppose that there exist non-empty words $f, g \in \Xp\Xo$ and scalars $k_{f}, k_{g} \in \K$ such that
\begin{align}\label{formula:wr-fwr-g-I}
p_{\beta} := \beta + \sum_{f} k_{f} \beta f + \sum_{g} k_{g} g \in I, \qquad \text{with } C(m_{f}) <_{O} C(m_{\beta}) \text{ and } g_{L} <_{\X} \beta.
\end{align} 
Then there exist non-empty words $h \in \Xp\Xo$, non-empty words $q \in \Xo$, and scalars $k_{q}, k_{h} \in \K$ such that
\begin{align}\label{formula:beta-betaq-h-I}
\beta + \sum_{q} k_{q} \beta q +  \sum_{h} k_{h} h \in I, \qquad \text{with } h_{L} <_{\X} \beta.
\end{align}
\end{lemma}

\begin{proof}
Let us define the index set of the first summation and its maximal component degree by
\begin{align*}
W_{1} := \{f \mid p_{\beta} \in I \text{ and } k_{f} \neq 0\}, \qquad d_{1} := \max \{\deg(f) \mid f \in W_{1}\}.
\end{align*}
If $d_{1} = 0$, the assertion holds immediately. Assume therefore that $d_{1} \ge 1$. Recall from Lemma~\ref{lem:comultiplication-words-2} that there exist non-empty words $\beta', f', g' \in \Xp\Xo$ and polynomials $\beta'', f'', g'' \in \K\X^{+}$ such that applying the comultiplication yields:
\begin{align*}
&\ \Delta\left(\beta + \sum_{f \in W_{1}} k_{f} \beta f + \sum_{g} k_{g} g\right) \\
\in & \ \beta \otimes 1 + 1 \otimes \beta + \sum \beta' \otimes \beta''  \\
& + \sum_{f \in W_{1}} k_{f} \left(\beta \otimes 1 + 1 \otimes \beta + \sum \beta' \otimes \beta''\right)\left(f \otimes [1-r_{f}] + 1 \otimes f + \sum f' \otimes f''\right) \\
& + \sum k_{g} \left(g \otimes [1-r_{g}] + 1 \otimes g + \sum g' \otimes g''\right) + I \otimes \K\X \\
\subseteq & \ \beta \otimes 1 + 1 \otimes \beta + \sum \beta' \otimes \beta'' \\
& + \sum_{f \in W_{1}} k_{f} \Big(\beta f \otimes [1-r_{f}] + \beta \otimes f + \sum \beta f' \otimes f'' \\
& + f \otimes \beta [1-r_{f}] + 1 \otimes \beta f + \sum f' \otimes \beta f'' \\
& + \sum \beta' f \otimes \beta'' [1-r_{f}] + \sum \beta' \otimes \beta'' f + \sum \sum \beta' f' \otimes \beta'' f''\Big) \\
& + \sum k_{g} \left(g \otimes [1-r_{g}] + 1 \otimes g + \sum g' \otimes g''\right) + I \otimes \K\X,
\end{align*}
where $C((\beta')_{L}) <_{\lOr} C(\beta)$, $C(m_{f'}) \le_{O} C(m_{f}) <_{O} C(m_{\beta})$, $(g')_{L} \le_{\X} g_{L} <_{\X} \beta$, and $\deg(f') < d_{1}$.

Let $p_{1} := 1 + \sum_{f\in W_{1}} k_{f} f$. Fix an element $f_{0} \in W_{1}$, and partition the set $W_{1}$ into
\begin{align*}
W_{f_{0}} := \{f \in W_{1} \mid [1-r_{f}] \in [1-r_{f_{0}}] + I\}, \qquad W_{1}' := W_{1} \setminus W_{f_{0}}.
\end{align*}
We now proceed by analyzing the following two cases.

\medskip
\noindent
\textit{Case 1}: Suppose $[1-r_{f_{0}}] \in \K p_{1} + I$. By applying the canonical augmentation map $\epsilon$, we deduce that $p_{1} \in [1-r_{f_{0}}] + I$. Substituting this relation back into our initial hypothesis directly yields
\begin{align*}
\beta[1-r_{f_{0}}] + \sum_{g} k_{g} g \in I,
\end{align*}
which provides the desired form.

\medskip
\noindent
\textit{Case 2}: Suppose $[1-r_{f_{0}}] \notin \K p_{1} + I$. Then the sum $\K p_{1} + I + \K [1-r_{f_{0}}]$ is direct. We can therefore decompose the free algebra as $\K\X = \K p_{1} \oplus I \oplus \K [1-r_{f_{0}}] \oplus U$ for some choosing complement subspace $U \subseteq \K\X$. We define a linear functional $\phi: \K\X \rightarrow \K$ by setting:
\begin{align*}
\phi(p_{1}) = 1, \quad \phi(I) = \phi(U) = 0, \quad \text{and} \quad \phi([1-r_{f_{0}}]) = 0.
\end{align*}

Applying the linear map $\id \otimes \phi$ to the containment $\Delta(p_{\beta}) \in I \otimes \K\X + \K\X \otimes I$, we find that there exist non-empty words $g_{1} \in \Xp\Xo$ and scalars $k_{g_{1}} \in \K$ such that
\begin{align}\label{formula:beta-betaf-g1}
\beta + \sum_{f \in W_{1}'} k_{f} \beta f \phi([1-r_{f}]) + \sum_{f\in W_{1}}\sum k_{f} \beta f' \phi(f'') + \sum_{g_{1}} k_{g_{1}} g_{1} \in I, \qquad \text{with } (g_{1})_{L} <_{\X} \beta. 
\end{align}
By construction, the strict inequality $|W_{1}'| < |W_{1}|$ holds. By iterating this reduction process, we can systematically eliminate the second term in equation~\eqref{formula:beta-betaf-g1}. This ensures the existence of non-empty words $w, g_{2} \in \Xp\Xo$ and scalars $k_{w}, k_{g_{2}} \in \K$ such that
\begin{align}\label{formula:beta-betaf-g2}
\beta + \sum_{w} k_{w} \beta w + \sum_{g_{2}} k_{g_{2}} g_{2} \in I, 
\end{align}
where $\deg(w) < d_{1}$, $C(m_{w}) <_{O} C(m_{\beta})$, and $(g_{2})_{L} <_{\X} \beta$.

Now, let us define the new supporting index set:
\begin{align*}
W_{2} := \{w \mid w \text{ appears in } \eqref{formula:beta-betaf-g2} \text{ and } k_{w} \neq 0\}, \qquad d_{2} := \max \{\deg(w) \mid w \in W_{2}\}.
\end{align*}
Since $d_{2} < d_{1}$, a standard induction argument on the degree $d_{1}$ allows us to transition from equation~\eqref{formula:beta-betaf-g2} to the final required relation~\eqref{formula:beta-betaq-h-I}, completing the proof.
\end{proof}

Recall the canonical projection $\pi:\K\X \to H$ defined in Lemma~\ref{lem:X-generates-H}. For any $p \in \K\X$, we denote its image under $\pi$ by $\overline{p} := \pi(p)$.

Utilizing Lemma~\ref{lem:beta-qbeta-h-I}, we establish a stronger reduction result for prime words:

\begin{lemma}\label{lem:beta-u-I}
Let $\beta$ be a prime word in $\Xp$. Suppose that there exist non-empty words $h \in \Xp\Xo$, $q \in \Xo$, and scalars $k_{h}, k_{q} \in \K$ such that
\begin{align*}
\beta + \sum_{q} k_{q} \beta q + \sum_{h} k_{h} h \in I, \qquad \text{with } h_{L} <_{\X} \beta.
\end{align*}
Then there exist non-empty words $\mu \in \Xp\Xo$ and scalars $k_{\mu} \in \K$ such that 
\begin{align*}
\beta \in \sum k_{\mu} \mu + I, \qquad \text{with } \mu_{L} <_{\X} \beta.
\end{align*}
\end{lemma}

\begin{proof}
Applying the canonical projection $\pi$ to the given relation yields the following identity in $H$:
\begin{align*}
\overline{\beta} \left(\overline{1} + \sum_{q} k_{q} \overline{q}\right) + \sum_{h} k_{h} \overline{h} = 0.
\end{align*}
Let us define the right-hand factor and its coproduct decomposition by
\begin{align*}
v := \overline{1} + \sum_{q} k_{q} \overline{q}, \quad r := \rank(\Delta_{H}(v)), \quad \Delta_{H}(v) = \sum_{i=1}^{r} v'_{i} \otimes v''_{i},
\end{align*}
and denote the set of right tensor factors by $V'' := \{v''_{1}, \ldots, v''_{r}\}$. By Lemma~\ref{lem:comultiplication-words-2}, there exist non-empty words $\beta', h' \in \Xp\Xo$ and polynomials $\beta'', h'', r_{h} \in \K\X^{+}$ such that applying the coproduct $\Delta_{H}$ gives:
\begin{align*}
& \  \Delta_{H}\left(\overline{\beta} v + \sum_{h} k_{h} \overline{h}\right) \\
=& \  \left(\overline{\beta} \otimes \overline{1} + \overline{1} \otimes \overline{\beta} + \sum \overline{\beta'} \otimes \overline{\beta''}\right)\left(\sum_{i=1}^{r} v'_{i} \otimes v''_{i}\right) + \sum_{h} k_{h} \left(\overline{h} \otimes \overline{[1 - r_{h}]} + \sum \overline{1} \otimes \overline{h} + \sum \overline{h'} \otimes \overline{h''}\right)\\
=& \  \sum_{i=1}^{r} \overline{\beta} v'_{i} \otimes v''_{i} + \sum_{i=1}^{r} v'_{i} \otimes \overline{\beta} v''_{i} + \sum_{i=1}^{r} \overline{\beta'} v'_{i} \otimes \overline{\beta''} v''_{i} \\
& + \sum_{h} k_{h} \left(\overline{h} \otimes \overline{[1 - r_{h}]} + \sum \overline{1} \otimes \overline{h} + \sum \overline{h'} \otimes \overline{h''}\right),
\end{align*}
where $(\beta')_{L} <_{\X} \beta$, and $(h')_{L} \le_{\X} h_{L} <_{\X} \beta$.

Recall from \cite[Lemma 1.2.2]{Ra2012} that the set $V''$ is linearly independent in $H$. We can therefore decompose the Hopf algebra as $H = \left(\bigoplus_{i=1}^{r} \K v''_{i}\right) \oplus U$ for a choice of complement subspace $U \subseteq H$. For each $v''_{i} \in V''$, we define a linear functional $\phi_{v''_{i}} : H \rightarrow \K$ by setting:
\begin{align*}
\phi_{v''_{i}} (v''_{j}) = \delta_{i, j} \quad \text{for } v''_{j} \in V'', \qquad \text{and} \qquad \phi_{v''_{i}}(U) = 0. 
\end{align*}
Applying the linear map $\id \otimes \phi_{v''_{i}}$ to the identity $\Delta_{H}\left( \overline{\beta} v + \sum k_{h} \overline{h}\right) = 0$, we find that there exist, depending on $v''_{i}$, elements $\eta \in \Xp^{+}$, $u, y \in H_{[0]}$, and scalars $k_{\overline{\eta} u}, k_{y} \in \K$ such that 
\begin{align*}
\overline{\beta} v'_{i} + \sum k_{\overline{\eta} u} \overline{\eta} u + \sum k_{y} y = 0, \qquad \text{with } \eta_{L} <_{\X} \beta.
\end{align*}   
Right-multiplying this equation by $S_{H}(v''_{i})$ and summing over all $i$ yields:
\begin{align*}
\sum_{i=1}^{r} \overline{\beta} v'_{i} S_{H}(v''_{i}) + \sum_{i=1}^{r} \sum k_{\overline{\eta} u} \overline{\eta} u S_{H}(v''_{i}) + \sum_{i=1}^{r} \sum k_{y} y S_{H}(v''_{i}) = 0.
\end{align*}
Since $\epsilon_{H}(v) = 1$, the antipode property guarantees that $\sum_{i=1}^{r} v'_{i}S_{H}(v''_{i}) = 1$. The identity simplifies to:
\begin{align*}
\overline{\beta} + \sum_{i=1}^{r} \sum k_{\overline{\eta} u} \overline{\eta} u S_{H}(v''_{i}) + \sum_{i=1}^{r} \sum k_{y} y S_{H}(v''_{i}) = 0.
\end{align*}
Lifting this back to the free algebra and applying the augmentation map $\epsilon$, we deduce the existence of words $z_{1} \in \Xo$, $z_{2} \in \Xo^{+}$, and scalars $k_{\eta z_{1}}, k_{z_{2}} \in \K$ such that 
\begin{align*}
\beta + \sum k_{\eta z_{1}} \eta z_{1} + \sum k_{z_{2}} z_{2} \in I,
\end{align*}
where $(\eta z_{1})_{L} = \eta_{L} <_{\X} \beta$, and $(z_{2})_{L} <_{\X} \beta$. Setting $\mu$ to be the words appearing in the lower-order terms completes the proof.
\end{proof}

Now, we demonstrate that any reducible letter of positive degree can be factored modulo $I$ into a linear combination of products of letters of strictly smaller order.

\begin{lemma}\label{lem:rl-sum-smaller-irrls}
Let $x$ be an $I$-reducible letter in $X_{+}$. Then there exist letters $x_{1}, \ldots, x_{n} \in X$ and scalars $k_{x_{1}\cdots x_{n}} \in \K$ such that 
\begin{align*}
x \in \sum k_{x_{1}\cdots x_{n}} x_{1} \cdots x_{n} + I, \qquad \text{with } x_{1}, \ldots, x_{n} <_{X} x.
\end{align*}
\end{lemma}

\begin{proof}
If $x \in I$, the assertion holds trivially. Assume therefore that $x \notin I$. We proceed via case analysis based on the structural type of the letter $x$.

\medskip
\noindent
\textit{Case 1}: Suppose $x \in O_{+}$. By Lemma~\ref{lem:zw-wz-I} and application of the canonical augmentation map $\epsilon$, there exist non-empty words $\omega \in \Xp \Xo$ and scalars $k_{\omega} \in \K$ such that
\begin{align*}
x \in \sum k_{\omega} \omega + I, \qquad \text{with } \omega <_{\X} x,
\end{align*}
or equivalently, by Corollary~\ref{cor:well-orders-compatibility-2},
\begin{align*}
x \in \sum k_{\omega} \omega + I, \qquad \text{with } C(r(\omega)) <_{O} x.
\end{align*}
Let us define the support set and its associated prime word set by
\begin{align*}
W := \{\omega \mid x \in \sum k_{\omega} \omega + I \text{ and } k_{\omega} \neq 0\}, \qquad W_{L} := \{\omega_{L} \mid \omega \in W\}.
\end{align*}
Let $\beta$ denote the greatest word of $W_{L}$ with respect to the order $<_{\X}$. We can thus isolate the leading structural components of $x$ as
\begin{align*}
x \in \sum_{\beta\alpha \in W} k_{\beta \alpha} \beta \alpha + \sum_{\substack{\omega \in W \\ \omega_{L} <_{\X} \beta}} k_{\omega} \omega + I, \qquad \text{where } \beta \in \Xp,\ C(r(\beta\alpha)), C(r(\omega)) <_{O} x.
\end{align*}
Without loss of generality, we may assume that the leading sum satisfies $\sum_{\beta \alpha \in W} k_{\beta \alpha} \beta \alpha \notin I$. 

Since $(\X, <_{\X})$ is well-ordered by Lemma~\ref{lem:well-orders}, we can proceed by induction on $\beta$ with respect to $<_{\X}$. If $\beta <_{\X} x$, then Lemma~\ref{lem:properties-alpha_R}~(e) implies that for every word $\omega \in W$ and each constituent letter $b$ of $\omega$, the inequality $b \le_{\X} \omega_{L} \le_{\X} \beta <_{\X} x$ holds. Lemma~\ref{lem:compatibility-orders} then ensures that $b <_{X} x$, completing this subcase.

Conversely, assume that 
\begin{align*}
\beta \ge_{\X} x \qquad (\text{which translates to } C(r(\beta)) \ge_{O} x \text{ by Corollary~\ref{cor:well-orders-compatibility-2}}).
\end{align*}
This implies that $\alpha \neq 1$ for every word $\beta \alpha \in W$, because $C(r(\beta\alpha)) <_{O} x$.

Let $\ell$ denote the minimal non-empty word in $\Xp$, which corresponds to the least letter in $X_{+}$ under $<_{X}$ via Lemma~\ref{lem:x-least}~(c). If $\beta = \ell$, then $x \in \sum k_{\ell\alpha} \ell\alpha + \sum k_{\omega} \omega + I$, where $\alpha \in \Xo$ and $\omega \in \Xo^{+}$. Since $\ell \ge_{\X} x$ and $x \in X_{+}$, it must be that $x = \ell$, which gives $\ell - \sum k_{\ell\alpha} \ell\alpha - \sum k_{\omega} \omega \in I$. Lemma~\ref{lem:beta-u-I} then guarantees the existence of non-empty words $u \in \Xp\Xo$ and scalars $k_{u} \in \K$ such that 
\begin{align*}
\ell \in \sum k_{u} u + I, \qquad \text{with } u_{L} <_{\X} \ell.
\end{align*}
This forces $u \in \Xo^{+}$, matching the claim.

Now, suppose $\beta >_{\X} \ell$. By Lemma~\ref{lem:comultiplication-words-2} and the comultiplication on $X$, there exist elements $x'_{+} \in X_{+}^{*}$, $x'_{0} \in X_{0}^{*}$, non-empty words $(\beta\alpha)', \omega' \in \Xp\Xo$, and polynomials $x'', (\beta\alpha)'', \omega'' \in \K\X^{+}$ such that 
\begin{align*}
& \ \Delta\left( \sum_{\beta\alpha \in W} k_{\beta\alpha} \beta\alpha + \sum_{\substack{\omega \in W \\ \omega_{L} <_{\X} \beta}} k_{\omega} \omega - x \right) \\
\in & \sum_{\beta\alpha \in W} k_{\beta\alpha} \left(\beta\alpha \otimes [1-r_{\alpha}] + \beta \otimes \alpha + 1 \otimes \beta\alpha + \sum (\beta\alpha)' \otimes (\beta\alpha)''\right) \\
& + \sum_{\substack{\omega \in W \\ \omega_{L} <_{\X} \beta}} k_{\omega}\left(\omega \otimes [1-r_{\omega}] + 1 \otimes \omega + \sum \omega' \otimes \omega''\right) \\
& - \left(x \otimes 1 + 1 \otimes x + \sum x'_{+}x'_{0} \otimes x''\right) + I \otimes \K\X,
\end{align*}
where $x'_{+}x'_{0} \neq 1$, $(\beta\alpha)' \neq 1$, $C((x'_{+}x'_{0})_{L}) = C(x'_{+}) <_{O} C(x)$, $(\omega')_{L} \le_{\X} \omega_{L} <_{\X} \beta$, and either $C(((\beta\alpha)')_{L}) <_{\lOr} C(\beta)$ or $(\beta\alpha)' = \beta c$ for some $c \in \X^{+}$ satisfying $C(m_{c}) <_{O} C(m_{\beta})$.

Note that the subspace $\K (\sum_{\beta\alpha \in W} k_{\beta\alpha} \alpha) + I$ forms a direct sum because $\sum_{\beta\alpha \in W} k_{\beta\alpha}\alpha \notin I$. Since $\alpha \neq 1$ for all $\beta\alpha \in W$, the application of the augmentation map $\epsilon$ shows that the sum $\K (\sum_{\beta\alpha \in W} k_{\beta\alpha} \alpha) + I + \K 1$ is also direct. We can therefore choose a complement decomposition of the free algebra of the form $\K\X = \K (\sum_{\beta\alpha \in W} k_{\beta\alpha} \alpha) \oplus I \oplus \K 1 \oplus V$ for some subspace $V \subseteq \K\X$. 
Define a linear functional $\varphi : \K\X \rightarrow \K$ by setting:
\begin{align*}
\varphi\left(\sum_{\beta\alpha \in W} k_{\beta\alpha} \alpha\right) = 1, \quad \varphi(1) = 0, \quad \text{and} \quad \varphi(I) = \varphi(V) = 0.
\end{align*}
Applying the linear map $\id \otimes \varphi$ to the containment relation for the coproduct inside $I \otimes \K\X + \K\X \otimes I$, we find that there exist non-empty words $f, g \in \Xp\Xo$ and scalars $k_{f}, k_{g} \in \K$ satisfying
\begin{align*}
\beta + \sum k_{f} \beta f + \sum k_{g} g \in I, \qquad \text{with } C(m_{f}) <_{O} C(m_{\beta}) \text{ and } g_{L} <_{\X} \beta.
\end{align*}

By invoking Lemmas~\ref{lem:beta-qbeta-h-I} and \ref{lem:beta-u-I}, there exist non-empty words $\mu \in \Xp\Xo$ and scalars $k_{\mu} \in \K$ such that 
\begin{align*}
\beta \in \sum k_{\mu} \mu + I, \qquad \text{with } \mu_{L} <_{\X} \beta.
\end{align*}
We can now rewrite the element $x$ modulo $I$ as follows:
\begin{align*}
x \in \sum_{\beta\alpha \in W} \sum_{\mu} k_{\beta \alpha} k_{\mu} \mu \alpha + \sum_{\substack{\omega \in W \\ \omega_{L} <_{\X} \beta}} k_{\omega} \omega + I, 
\end{align*}
where $C(r( \mu \alpha)), C(r(\omega)) <_{O} x$, and $\mu_{L} <_{\X} \beta$. Since $\mu_{L} <_{\X} \beta$ and $C(m_{\alpha}) <_{O} C(m_{\beta})$, it follows that $(\mu\alpha)_{L} <_{\X} \beta$. Applying Lemma~\ref{lem:zw-wz-I}, we establish that there exist non-empty words $\nu \in \Xp\Xo$ and scalars $k_{\nu} \in \K$ such that 
\begin{align*}
x \in \sum k_{\nu} \nu + I, \qquad \text{with } \nu_{L} <_{\X} \beta.
\end{align*}
By our induction hypothesis on $\beta$ (since $\nu_{L} <_{\X} \beta$), there exist letters $x_{1}, \ldots, x_{n} \in X$ and scalars $k_{x_{1}\cdots x_{n}} \in \K$ such that
\begin{align*}
x \in \sum k_{x_{1}\cdots x_{n}} x_{1} \cdots x_{n} + I, \qquad \text{with } x_{1}, \ldots, x_{n} <_{X} x.
\end{align*}

\medskip
\noindent
\textit{Case 2}: Suppose $x_{a} \in X_{+} \setminus O_{+}$ with $x \in O_{+}$ and $a \in O_{0}$. 

By Lemma~\ref{lem:zw-wz-I} and the evaluation under $\epsilon$, we have $x_{a} \in \sum k_{\omega} \omega + I$ for some non-empty words $\omega \in \Xp\Xo$ satisfying $\omega <_{\X} x_{a}$, and scalars $k_{\omega} \in \K$. By definition of the order $<_{\X}$, there exist letters $a_{1}, \ldots, a_{s} \in O_{0}$ satisfying $a_{1} <_{O} \cdots <_{O} a_{s} <_{O} a$ such that
\begin{align*}
x_{a} \in k_{x} x + \sum_{1 \le i \le s} k_{x_{a_{i}}} x_{a_{i}} + \sum k_{\omega} \omega + I, \qquad \text{with } \omega <_{\X} x.
\end{align*}
An analogous reduction argument to Case 1 applies here by tracking the behavior of $\Delta(x_{a} - k_{x} x - \sum_{1 \le i \le s} k_{x_{a_{i}}} x_{a_{i}})$ instead of $\Delta(x)$. It follows that there exist letters $y_{1}, \ldots, y_{m} \in X$ and scalars $k_{y_{1}\cdots y_{m}} \in \K$ such that
\begin{align*}
x_{a} \in k_{x} x + \sum_{1 \le i \le s} k_{x_{a_{i}}} x_{a_{i}} + \sum k_{y_{1}\cdots y_{m}} y_{1} \cdots y_{m} + I,
\end{align*}
where $y_{1}, \ldots, y_{m} <_{X} x <_{X} x_{a_{1}} <_{X} \cdots <_{X} x_{a_{s}} <_{X} x_{a}$. This completes the proof.
\end{proof}

We show that the original element of every irreducible letter is itself irreducible.

\begin{lemma}\label{lem:x-reducible-x_a-reducible}
Let $x \in O$. Then the following assertions hold:
\begin{itemize}
\item [$\mathrm{(a)}$] If $x$ is $I$-reducible, then $x_{a}$ is $I$-reducible for every element $a \in O_{0}$.
\item [$\mathrm{(b)}$] $C(X_{I,+}) = O_{I,+}$.
\end{itemize}
\end{lemma}

\begin{proof}
Assume to the contrary that there exist an $I$-reducible letter $x \in O$ and a letter $a \in O_{0}$ such that $x_{a}$ is $I$-irreducible. Since $x$ is assumed to be $I$-reducible, Lemma~\ref{lem:zw-wz-I} and  Corollary~\ref{cor:well-orders-compatibility-2} guarantee the existence of non-empty words $\omega \in \Xp\Xo$ and scalars $k_{\omega} \in \K$ such that 
\begin{align*}
x \in \sum k_{\omega} \omega + I, \qquad \text{with } C(r(\omega)) <_{O} x.
\end{align*}
By utilizing the assignment $f_{X}$ constructed in Lemma~\ref{lem:X-generates-H}, we can express $x_{a}$ modulo $I$ as $x_{a} \in \sum k_{a_1, a_2} a_{1} x a_{2} + I$ for some elements $a_{1}, a_{2} \in X_{0}^{*}$ and scalars $k_{a_1, a_2} \in \K$. This yields a polynomial relation of the form
\begin{align*}
p := x_{a} - \sum k_{\omega} k_{a_1, a_2} a_{1} \omega a_{2} \in I.
\end{align*}
We now evaluate the rightmost factors of the terms appearing in the summation:
\begin{itemize}
\item If $a_{2} \neq 1$, then $C(r(a_{1}\omega a_{2})) = a_{2} <_{O} x$.
\item If $a_{2} = 1$, then $C(r(a_{1}\omega)) = C(r(\omega)) <_{O} x$.
\end{itemize}
In either case, Corollary~\ref{cor:well-orders-compatibility-2} implies that $a_{1} \omega a_{2} <_{\X} x <_{\X} x_{a}$. Consequently, the leading word of the polynomial satisfies $\mathrm{LW}(p) = x_{a}$, which forces $x_{a}$ to be $I$-reducible, yielding a contradiction. 

We now establish assertion (b). From the validity of Part (a), the inclusion $C(X_{I,+}) \subseteq O_{I,+}$ follows directly. Conversely, since $O_{I,+} = C(O_{I,+}) \subseteq C(X_{I,+})$, we obtain the reverse containment. Hence, the equality $C(X_{I,+}) = O_{I,+}$ holds.
\end{proof}

Now, we show that every non-empty word can be decomposed as a linear combination of products consisting exclusively of degree-zero letters and positive-degree irreducible letters.

\begin{theorem}\label{thm:rl-sum-lower-irrls}
With the notation established in Lemma~\ref{lem:X-generates-H}, if $x$ is an $I$-reducible letter in $X_{+}$, then there exist letters $x_{1}, \ldots, x_{n} \in X_{I,+} \cup X_{0}$ and scalars $k_{x_{1}\cdots x_{n}} \in \K$ such that
\begin{align*}
x \in \sum k_{x_{1}\cdots x_{n}} x_{1} \cdots x_{n} + I, \qquad \text{with } x_{1}, \ldots, x_{n} <_{X} x.
\end{align*}
Consequently, every non-empty word in $\X$ can be expressed modulo $I$ as a linear combination of products of letters from $X_{I,+} \cup X_{0}$, or more precisely, from $O_{I,+} \cup X_{0}$.
\end{theorem}

\begin{proof}
Since $(X, <_{X})$ is well-ordered, we proceed by induction on reducible letters of $X_{+}$ with respect to the order $<_{X}$. By Lemma~\ref{lem:rl-sum-smaller-irrls}, there exist constituent letters $x_{1}, \ldots, x_{n} \in X$ and scalars $k_{x_{1}\cdots x_{n}} \in \K$ such that
\begin{align*}
x \in \sum k_{x_{1}\cdots x_{n}} x_{1} \cdots x_{n} + I, \qquad \text{with } x_{1}, \ldots, x_{n} <_{X} x.
\end{align*}
Applying the induction hypothesis to the terms on the right-hand side (since $x_{i} <_{X} x$ for all $i$), there exist replacement letters $y_{1}, \ldots, y_{m} \in X_{I,+} \cup X_{0}$ and scalars $k_{y_{1} \cdots y_{m}} \in \K$ yielding:
\begin{align*}
x \in \sum k_{y_{1}\cdots y_{m}} y_{1} \cdots y_{m} + I, \qquad \text{with } y_{1}, \ldots, y_{m} <_{X} x.
\end{align*}
Thus, every reducible letter in $X_{+}$ can be expressed modulo $I$ as a linear combination of products of letters in $X_{I,+} \cup X_{0}$. Since irreducible letters in $X_{+}$ trivially satisfy the representation, it follows that all letters in $X_{+}$ share this property. 

By referencing the structural map $f_{X}$ defined in Lemma~\ref{lem:X-generates-H}, every letter in $X_{+}$ can be written modulo $I$ as a linear combination of products of elements from $C(X_{I,+}) \cup X_{0}$. By Lemma~\ref{lem:x-reducible-x_a-reducible}, this is precisely the set $O_{I,+} \cup X_{0}$. Extending this result factor by factor, it follows that any non-empty word in $\X$ can be expressed modulo $I$ as a linear combination of products of letters from $O_{I,+} \cup X_{0}$, completing the proof.
\end{proof}

Applying the foundational results established above to the class of Hopf algebras whose Hopf coradical forms a Hopf subalgebra, we prove that a left or right Noetherian Hopf algebra is affine provided its Hopf coradical is affine.

\begin{theorem}\label{thm:main-theorem2}
Let $H$ be a Hopf algebra over $\K$. Then the following assertions hold:
\begin{itemize}
    \item [$\mathrm{(a)}$] If $S_{H}(H_{[0]}) \subseteq H_{[0]}$, then there exists a generating set $X$ of $H$ as described in Lemma~\ref{lem:X-generates-H} such that the set $\pi(O_{I,+}) \cup G_{[0]}$ generates $H$ as an algebra, where $\pi: \K\X \rightarrow H$ is the canonical projection, $I = \ker\pi$, and $G_{[0]}$ is a choice of generating set for $H_{[0]}$.
    \item [$\mathrm{(b)}$] If $H$ is left or right Noetherian and $H_{[0]}$ is affine, then $H$ is affine.
\end{itemize}
\end{theorem}

\begin{proof}
Assertion (a) follows directly from Lemma~\ref{lem:X-generates-H} and Theorem~\ref{thm:rl-sum-lower-irrls}. 

For assertion (b), first assume that $H$ is left Noetherian. By Lemma~\ref{lem:Sk2006}, the Hopf coradical $H_{[0]}$ is a Hopf subalgebra of $H$. Since $H_{[0]}$ is assumed to be affine, we may select a finite generating set $G_{[0]}$ for $H_{[0]}$. By virtue of Part (a), there exists a generating set $X$ of $H$ satisfying the conditions of Lemma~\ref{lem:X-generates-H} such that $\pi(O_{I,+}) \cup G_{[0]}$ generates $H$ as an algebra. Since $H_{[0]}$ is left Noetherian, Lemma~\ref{lem:O_I-finite} guarantees that the set $\pi(O_{I,+}) \cup G_{[0]}$ is finite, which implies that $H$ is affine.

Conversely, if $H$ is right Noetherian, then its opposite-coopposite Hopf algebra $H^{\mathrm{op},\mathrm{cop}}$ is left Noetherian and consequently affine by the preceding argument. Since affineness is preserved under passing to the opposite-coopposite structure, it follows that $H$ is also affine.
\end{proof}

\begin{remark}\label{rmk:more-result}
Let $H$ be a Hopf algebra, and let $K$ be a Hopf subalgebra of $H$ containing the coradical, i.e., $H_{(0)} \subseteq K$. Given a generating set $G_{K}$ of $K$, a straightforward application of the proof of Theorem~\ref{thm:main-theorem2} combined with the generalized lifting method establishes the existence of a finite set $Q \subseteq \ker \epsilon_{H}$ such that $Q \cup G_{K}$ generates $H$ as an algebra. 

As a consequence, $H$ is affine if $H$ is left or right Noetherian and $K$ is affine. In other words, a left or right Noetherian Hopf algebra is affine whenever its coradical is contained in an affine Hopf subalgebra.
\end{remark}

Since any Hopf subalgebra $H$ is faithfully flat over every Hopf subalgebra $K \supseteq H_{(0)}$, the algebra $K$ inherits the Noetherian property from $H$. Combining this with Theorem~\ref{thm:main-theorem2}, we obtain the following corollary:

\begin{corollary}\label{cor:Hopfsubalg-affine-1}
Let $H$ be a left or right Noetherian Hopf algebra over $\K$. Then every ascending chain of Hopf subalgebras of $H$ containing $H_{(0)}$ stabilizes.
\end{corollary}

\section{A Noetherian Hopf Algebra is Affine Iff Its Hopf Coradical is Affine}\label{sec:6}

In this section, we establish that a left or right Noetherian Hopf algebra is affine if and only if its Hopf coradical is affine. To achieve this, we first demonstrate that the Hopf coradical of an affine Hopf algebra must be affine.

We employ the notions of mirror reduction order, reduction-factorization and prime words from Section~\ref{sec:2}, and rebuild the overall framework accordingly. We also introduce some notation and adjust the underlying set structure.

Let $O$ be a set equipped with a degree map $t_{O}: O \rightarrow \mathbb{N}$. For each 
$z\in O_{0}$, we choose pairwise disjoint sets of formal symbols $D_{z}$ such that $z \in D_{z}$, and define
\begin{align*}
Y := \left(\bigcup_{z \in O_{0}} D_{z} \right) \cup O_{+}.
\end{align*}
Define a map $D: Y \rightarrow O$ by 
\begin{align*}
D(d) = z, \qquad D(y) = y, \qquad \text{for } d \in D_{z}, \ z \in O_{0}, \ y \in O_{+}.
\end{align*}
The map $D$ lifts to a monoid map $D: \Y \rightarrow \lOr$.

We equip $Y$ with a degree map $t_Y: Y  \rightarrow  \mathbb{N}$ extending $t_O$:
\begin{align*}
t_{Y}(d) = t_{O}(z) = 0, \qquad t_{Y}(y) = t_{O}(y), \qquad \text{for } z \in O_{0}, \ d \in D_{z}, \ y \in O_{+}.
\end{align*}
It follows that $Y_{0} = \bigcup_{z \in O_{0}} D_{z}$, and $Y_{+} = O_{+}$.

For each $n \ge 0$, let $<_{O_{n}}$ be an arbitrary well-ordering on $O_{n}$, and define a total order $<_{O}$ on $O$ as follows: for any $x,y \in O$, 
\begin{align*}
x <_{O} y \Longleftrightarrow 
\left\{\begin{array}{l}
t_{O}(x) < t_{O}(y), \text{ or } \\
t_{O}(x) = t_{O}(y) = n \text{ and } x <_{O_{n}} y.
\end{array}\right.
\end{align*}
For each $z \in O_{0}$, let $<_{D_z}$ be an arbitrary well-ordering on $D_{z}$, and define a total order $<_{Y}$ on $Y$: for any $x, y \in Y$, 
\begin{align*}
x <_{Y} y \Longleftrightarrow 
\left\{\begin{array}{l}
D(x) <_{O} D(y), \text{ or } \\
D(x) = D(y) = z \in O_{0} \text{ and } x <_{D_{z}} y.
\end{array}\right.
\end{align*}

The (mirror) reduction order $<_{\lOr}$ on $\lOr$ is defined as follows: for any non-empty words $\alpha, \beta \in \lOr$, set $1 <_{\lOr} \alpha$, and 
\begin{align*}
\alpha <_{\lOr} \beta \Longleftrightarrow 
& \left\{\begin{array}{l}
r(\alpha) <_{O} r(\beta), \text{ or } \\
r(\alpha) = r(\beta) \text{ and } 
\left\{\begin{array}{l}
|\alpha| < |\beta|, \text{ or } \\
|\alpha| = |\beta| \text{ and } \alpha <_{\mlex_O} \beta.
\end{array}\right.
\end{array}\right.
\end{align*}
The (mirror) reduction order $<_{\Y}$ on $\Y$ is defined as follows: for any words $\alpha, \beta \in \Y$, 
\begin{align*}
\alpha <_{\Y} \beta \Longleftrightarrow 
& \left\{\begin{array}{l}
D(\alpha) <_{\lOr} D(\beta), \text{ or } \\
D(\alpha) = D(\beta) \text{ and } \alpha <_{\mlex_{Y}} \beta.
\end{array}\right.
\end{align*}
Evidently, the sets $(O, <_{O})$, $(Y, <_{Y})$, $(\lOr, <_{\lOr})$, and $(\Y, <_{\Y})$ are well-ordered.

For a non-empty word $\alpha$ in $\Y$, the (mirror) reduction-factorization of $\alpha$, denoted by $\mathrm{mrf}(\alpha) = (\alpha_{L}, \alpha_{R})$, is the unique decomposition $\alpha = \alpha_{L}\alpha_{R}$
where $\alpha_{L}$ is the greatest prefix of $\alpha$ with respect to $<_{\Y}$. The word $\alpha$ is called prime if $\alpha_{L} = \alpha$. Let $m_{\alpha}$ denote the greatest letter occurring in $\alpha$ with respect to $<_{Y}$.

Suppose that $H$ is an affine Hopf algebra over $\K$. Then 
we may choose a finite set $O$ of formal symbols generating $H$ via an assignment $f_O: O \rightarrow H$ such that $f_O(O) \subseteq \ker \epsilon_H$. Define a degree map $t_O: O \rightarrow \mathbb{N}$ by
\[
t_O(y) = \min\{ n \in \mathbb{N} \mid f_O(y) \in H_{[n]} \}.
\]
This induces a well-ordered set $(O, <_{O})$ as constructed above.

By the local finiteness of coalgebras and the standard filtration of $H$, after refining $O$ if necessary, we may choose:
\begin{itemize}
\item pairwise disjoint finite sets $D_z$ of formal symbols, where $z \in O_0$;
\item an expanded finite set $Y = \left( \bigcup_{z \in O_0} D_z \right) \cup O_+$;
\item an assignment $f_Y: Y \rightarrow H$, lifting to the canonical algebra projection $\pi: \K\Y \rightarrow H$;
\item a degree map $t_Y: Y \rightarrow \mathbb{N}$ induced by $t_O$, equipping the free algebra $\K\Y$ with the grading $\deg(y) = t_{Y}(y)$ for each $y \in Y$;
\item a comultiplication $\Delta: \K\Y \rightarrow \K\Y \otimes \K\Y$,
\end{itemize}
such that all of the following conditions hold:
\begin{itemize}
\item $f_{Y}(Y) \subseteq \ker \epsilon_{H}$;
\item for each $z \in O_0$ and each $y \in D_z$, there exist elements $y' \in D_z$ and $y'', r_y \in \K D_z$ such that
\[
\Delta(y) = y \otimes (1 - r_y) + 1 \otimes y + \sum y' \otimes y'', \quad \text{ with } y' \neq y;
\]
\item for each $y \in Y_{n}$ with $n \ge 1$, there exist homogeneous elements $y' \in Y$, $y''\in \K Y$ and $r_{y} \in \K Y_{0}$ such that
\[
\Delta(y) = y \otimes (1-r_{y}) + 1 \otimes y + \sum y' \otimes y'', 
\]
where $t_{Y}(y') + t_{Y}(y'') \le n$, and $t_{Y}(y') < n$ (which guarantees $D(y') <_{O} D(y)$);
\item $\Delta_H \circ \pi = (\pi \otimes \pi) \circ \Delta$.
\end{itemize}

With the notation as above, we have the following comultiplication on words in $\Y$:
\begin{lemma}\label{lem:comultip-prime-words}
Let $\omega$ be a prime word in $\Y$ with $\deg(\omega) \ge 1$. Then there exist words $\omega' \in \Y$ and homogeneous polynomials $\omega'' \in \K\Y$ such that
\begin{align*}
\Delta(\omega) = \sum_{\substack{D(\omega')= D(\omega) \\ \deg(\omega'')=0}} \omega' \otimes \omega'' + \sum_{D(\omega') <_{\lOr} D(\omega)} \omega' \otimes \omega''.
\end{align*}
\end{lemma}

\begin{lemma}\label{lem:comultip-words}
Let $\omega$ be a non-prime word in $\Y$ with $\deg(\omega) \ge 1$, and let $\mathrm{mrf}(\omega) = (\omega_{L}, \omega_{R})$. Then
there exist words $(\omega_{L})', (\omega_{R})' \in \Y$ and homogeneous polynomials $(\omega_{L})'', (\omega_{R})'' \in \K\Y$ such that 
\begin{align*}
\Delta(\omega) = \sum_{\substack{D((\omega_{L})')= D(\omega_{L}) \\ \deg((\omega_{L})'')=0}} (\omega_{L})'(\omega_{R})' \otimes (\omega_{L})''(\omega_{R})'' + \sum_{D((\omega_{L})') <_{\lOr} D(\omega_{L})} (\omega_{L})'(\omega_{R})' \otimes (\omega_{L})''(\omega_{R})''.
\end{align*}
where the greatest letters satisfy $D(m_{(\omega_{R})'}) \le_{O} D(m_{\omega_{R}}) <_{O} D(m_{\omega_{L}})$.
\end{lemma}

Under the constructive framework described above, we are now ready to show that the Hopf coradical of an affine Hopf algebra is affine:

\begin{theorem}\label{thm:main-theorem3}
Let $H$ be an affine Hopf algebra over $\K$. Then the Hopf coradical $H_{[0]}$ is affine.
\end{theorem}

\begin{proof}
To demonstrate that $H_{[0]}$ is affine, it is sufficient to show the equality $H_{[0]} = \pi(\K\Yo)$. The containment $\pi(\K\Yo) \subseteq H_{[0]}$ holds by definition. For the reverse containment, let $C$ be a simple subcoalgebra of $H$, and let $M$ be a simple right $C$-subcomodule of $C$. Choose a non-zero element $z \in M$. We proceed via a case analysis on the counit evaluation of $z$ to demonstrate that $z \in \pi(\K\Yo)$.

\noindent
\textit{Case 1}: Suppose $\epsilon_{H}(z) = 0$. Let $r := \mathrm{rank}(\Delta_{H}(z)) - 2$. By the standard properties of tensor rank and counitality, there exist elements $z_{i}', z_{i}'', g, h \in M$ such that 
\begin{align*}
\Delta_{H}(z) = z \otimes g + h \otimes z + \sum_{i=1}^{r} z_{i}' \otimes z_{i}'', 
\end{align*}
where $\epsilon_{H}(z_{i}') = \epsilon_{H}(z_{i}'') = 0$, and $\epsilon_{H}(g) = \epsilon_{H}(h) = 1$.

Since $\pi(Y)$ forms a generating set of $H$, there exist words $\omega \in \Y$ and scalars $k_{\omega} \in \K$ such that $z = \pi\left(\sum k_{\omega} \omega\right)$. Let us define the support set and its corresponding prime word set by
\begin{align*}
W := \left\{\omega \;\middle|\; z = \pi\left(\sum k_{\omega} \omega\right) \text{ and } k_{\omega} \neq 0\right\}, \qquad W_{L} := \{D(\omega_{L}) \mid \omega \in W\}.
\end{align*}
Let $D(\zeta)$ denote the greatest word of $W_{L}$ with respect to the well-ordering $<_{\lOr}$. We expand $z$ in terms of this leading component as
\[
z = \pi \left(\sum_{\substack{\beta\alpha \in W \\ D(\beta) = D(\zeta)}} k_{\beta\alpha} \beta\alpha + \sum_{\substack{\omega \in W \\ D(\omega_{L}) <_{\lOr} D(\zeta)}} k_{\omega} \omega \right).
\]

We proceed by induction on $D(\zeta)$ with respect to $<_{\lOr}$. Applying the coproduct $\Delta_{H}$ yields:
\[
\Delta_{H}(z) = (\pi \otimes \pi) \circ \Delta \left(\sum_{\substack{\beta\alpha \in W \\ D(\beta) = D(\zeta)}} k_{\beta\alpha} \beta\alpha + \sum_{\substack{\omega \in W \\ D(\omega_{L}) <_{\lOr} D(\zeta)}} k_{\omega} \omega \right).
\]
If $\deg(\zeta) = 0$, then it follows that $\omega \in \Yo$ for all $\omega \in W$, which immediately yields $z \in \pi(\K\Yo)$. Now, assume $\deg(\zeta) \ge 1$. By Lemmas~\ref{lem:comultip-prime-words}~and~\ref{lem:comultip-words}, there exist words $\beta', \alpha', \omega' \in \Y$ and homogeneous polynomials $\beta'', \alpha'', \omega'' \in \K\Y$ such that
\begin{align}\label{formula:mainthm4_1}
& z \otimes g + h \otimes z + \sum_{i=1}^{r} z_{i}' \otimes z_{i}''  \\
= \ & 
(\pi \otimes \pi) \Bigg(\sum_{\substack{\beta\alpha \in W \\ D(\beta) = D(\zeta)}} k_{\beta\alpha} \bigg(\sum_{\substack{D(\beta') = D(\beta) \\ \deg(\beta'') =0}} \beta'\alpha' \otimes \beta''\alpha'' + \sum_{D(\beta') <_{\lOr} D(\beta)} \beta'\alpha' \otimes \beta''\alpha'' \bigg) \notag \\
& \qquad \qquad 
+ \sum_{\substack{\omega \in W \\ D(\omega_{L}) <_{\lOr} D(\zeta)}} k_{\omega} \sum \omega' \otimes \omega'' \Bigg), \notag
\end{align}
where $D((\omega')_{L}) \le_{\lOr} D(\omega_{L}) <_{\lOr} D(\zeta)$, $D(m_{\alpha'}) \le_{O} D(m_{\alpha}) <_{O} D(m_{\zeta})$, and for each summand $\alpha' \otimes \alpha''$ in $\Delta(\alpha)$, $\deg(\alpha') \le \deg(\alpha)$, with strict inequality precisely whenever $\deg(\alpha'') > 0$.

Let us define the maximal component degree $d_{1}$ and the right component subspace $U$ by
\begin{align*}
d_{1} &:= \max\{\deg(\alpha) \mid \beta\alpha \in W\}, \\
U &:=  \sum_{\substack{\beta\alpha \in W \\ D(\beta) = D(\zeta)}} \sum_{\substack{D(\beta') = D(\beta) \\ \deg(\beta'') = 0}}  \K \pi(\beta''\alpha'').
\end{align*}
We now branch into two distinct structural sub-cases:

\noindent
\textit{Case 1.1}: Suppose $d_{1} = 0$. This implies that $\alpha, \alpha' \in \Yo$ and $\alpha'' \in \K\Yo$ for all elements. We subdivide this condition as follows:

\noindent
\textit{Case 1.1.1}: Suppose $\K z \cap U \neq 0$. Since $\deg(\beta'') = 0$ for all $\beta''$ appearing in $U$, we have $\beta'' \in \K\Yo$, forcing $z \in \pi(\K\Yo)$.

\noindent
\textit{Case 1.1.2}: Suppose $\K z \cap U = 0$. The sum $\K z + U$ is therefore direct. We can decompose the entire Hopf algebra as $H = \K z \oplus U \oplus R$ for a choice of complement subspace $R \subseteq H$. Define a linear functional $\psi: H \rightarrow \K$ by setting:
\begin{align*}
\psi(z) = 1, \quad \psi(U) = 0, \text{ and }  \psi(R) = 0.
\end{align*}
Applying the linear map $\id \otimes \psi$ to equation~\eqref{formula:mainthm4_1}, we find that there exist a non-zero element $w \in M$, words $\mu \in \Y$, and scalars $k_{\mu} \in \K$ such that $w = \pi \left(\sum k_{\mu} \mu \right)$ with $D(\mu_{L}) <_{\lOr} D(\beta) = D(\zeta)$.

Since $M$ is a simple left $C^{*}$-module of $C$, there exists a linear form $c^{*} \in C^{*}$ such that $z = (\id \otimes c^{*})(\Delta_{H}(w))$. By Lemmas~\ref{lem:comultip-prime-words}~and~\ref{lem:comultip-words}, there exist words $\nu \in \Y$ and scalars $k_{\nu} \in \K$ satisfying $z = \pi\left(\sum k_{\nu} \nu\right)$ with $D(\nu_{L}) <_{\lOr} D(\zeta)$. By applying our induction hypothesis (since $D(\nu_{L}) <_{\lOr} D(\zeta)$), the claim follows.

\noindent
\textit{Case 1.2}: Suppose $d_{1} \ge 1$. Let us define the restricted target subspace 
\begin{align*}
V := \sum_{\substack{\beta\alpha \in W \\ D(\beta) = D(\zeta)}} \sum_{\substack{D(\beta') = D(\beta) \\ \deg(\beta'') = 0 \\ \deg(\alpha'') = 0}}  \K \pi(\beta''\alpha'').
\end{align*}

\noindent
\textit{Case 1.2.1}: If $\K z \cap V \neq 0$, an argument analogous to Case 1.1.1 applies directly.

\noindent
\textit{Case 1.2.2}: If $\K z \cap V = 0$, the sum $\K z + V$ is direct. We express the space as $H = \K z \oplus V \oplus T$ for a choice of complement subspace $T \subseteq H$. Define a linear functional $\varphi: H \rightarrow \K$ by setting:
\begin{align*}
\varphi(z) = 1, \quad \varphi(V) = 0, \text{ and } \varphi(T) = 0.
\end{align*}
Applying $\id \otimes \varphi$ to equation~\eqref{formula:mainthm4_1}, we obtain a non-zero element $v \in M$, words $\tau \in \Y$, and scalars $k_{\beta'\alpha'}', k_{\tau}' \in \K$ such that
\begin{align}\label{formula:mainthm4_11}
v =  \pi \left( \sum_{\substack{\beta\alpha \in W \\ D(\beta) = D(\zeta)}} \sum_{\substack{D(\beta')= D(\beta) \\ \deg(\beta'')=0 \\ \deg(\alpha'')>0}} k_{\beta'\alpha'}' \beta' \alpha' + \sum k_{\tau}' \tau \right),
\end{align}
where $D(\tau_{L}) <_{\lOr} D(\zeta)$, and $\deg(\alpha') < d_{1}$ (since $\deg(\alpha'') > 0$).

Since $M$ is a simple left $C^{*}$-module, it follows from Lemmas~\ref{lem:comultip-prime-words}~and~\ref{lem:comultip-words} that there exist words $\eta, f, \gamma \in \Y$ and scalars $k_{\eta f}'', k_{\gamma}'' \in \K$ such that
\begin{align*}
z = \pi\left(\sum k_{\eta f}'' \eta f + \sum k_{\gamma}'' \gamma \right), 
\end{align*}
where $D(\eta) = D(\zeta)$, $D(m_{f}) <_{O} D(m_{\zeta})$, $D(\gamma_{L}) <_{\lOr} D(\zeta)$, and $\deg(f) < d_{1}$.

Let us denote $d_{2} := \max\{\deg(f) \mid z = \pi\left(\sum k_{\eta f}'' \eta f + \sum k_{\gamma}'' \gamma \right) \}$. Note that $d_{2} < d_{1}$ holds strictly. Iterating this reduction process down the degree filtration yields elements $q \in \Yo$, $\theta, \xi \in \Y$, and scalars $k_{\theta q}''', k_{\xi}''' \in \K$ such that
\begin{align*}
z = \pi\left(\sum k_{\theta q}''' \theta q + \sum k_{\xi}''' \xi \right),
\end{align*}
where $D(\theta) = D (\zeta)$, $D(m_{q}) <_{O} D(m_{\zeta})$, and $D(\xi_{L}) <_{\lOr} D(\zeta)$.
The remaining verification follows precisely the same path as outlined in Case 1.1.

\medskip
\noindent
\textit{Case 2}: Suppose $\epsilon_{H}(z) = 1$ with $z \neq 1$. Let $r := \mathrm{rank}(\Delta_{H}(z)) - 1$. By tensor rank and counitality considerations, there exist elements $z_{i}', z_{i}'' \in M \cap \ker \epsilon_H$ such that
\begin{align*}
\Delta_{H}(z) = z \otimes z + \sum_{i=1}^{r} z_{i}' \otimes z_{i}''.
\end{align*}
The proof concludes by applying the exact reduction machinery of Case 1 to $z$ and adapting the sub-case analysis where necessary.
\end{proof}

\begin{remark}
(a) Since the proof of Theorem~\ref{thm:main-theorem3} does not require the existence of the antipode, the theorem also holds for affine bialgebras over $\K$.

\noindent
(b) Theorem~\ref{thm:main-theorem3} generalizes Zhuang's result that the coradical of an affine pointed Hopf algebra is affine (\cite[Corollary 3.5]{Z2013}).
\end{remark}

By combining Theorems~\ref{thm:main-theorem2} and \ref{thm:main-theorem3}, we deduce that a left or right Noetherian Hopf algebra is affine if and only if its Hopf coradical is affine. More generally, we establish the following equivalence criteria:

\begin{theorem}\label{thm:main-theorem4}
Let $H$ be a left or right Noetherian Hopf algebra over $\K$. Then the following conditions are equivalent: 
\begin{itemize}
\item [$\mathrm{(a)}$] $H$ is affine.
\item [$\mathrm{(b)}$] The coradical $H_{(0)}$ is contained in an affine Hopf subalgebra of $H$.
\item [$\mathrm{(c)}$] The Hopf coradical $H_{[0]}$ is affine.
\end{itemize}
In this case, any Hopf subalgebra of $H$ containing $H_{(0)}$ is also affine.
\end{theorem}

\begin{proof}
The implication (a) $\Rightarrow$ (b) holds trivially. The implication (b) $\Rightarrow$ (c) follows from Theorem~\ref{thm:main-theorem3}. The implication (c) $\Rightarrow$ (a) is guaranteed by Theorem~\ref{thm:main-theorem2}. 
Finally, any Hopf subalgebra containing $H_{(0)}$ inherits the Noetherian property from $H$, and hence is affine.
\end{proof}

Combining the above theorem with Molnar's results \cite{Mo1975}, we obtain the following sufficient conditions for a Noetherian Hopf algebra being affine.

\begin{corollary}\label{cor:when-Hopfalg-is-affine-6.7}
A left or right Noetherian Hopf algebra $H$ over $\K$ is affine provided that one of the following conditions holds:
\begin{itemize}
\item [(a)] $H_{(0)}$ is cocommutative;
\item[(b)] $H_{[0]}$ is commutative. 
\item [(c)] $H_{(0)}$ is finite-dimensional.
\end{itemize}
\end{corollary}

Note that Corollary \ref{cor:when-Hopfalg-is-affine-6.7} immediately implies that any left or right Noetherian pointed or copointed Hopf algebra is affine; for comprehensive treatments of the pointed framework, see \cite{GZ2017, JZ2025-2}.

\section{Noetherian Hopf Algebras with a Locally Affine Hopf Coradical}\label{sec:7}

By virtue of Theorem~\ref{thm:main-theorem4}, completing the verification of the affineness question reduces to showing that the Hopf coradical of a left or right Noetherian Hopf algebra is affine, or equivalently, that any Noetherian Hopf algebra generated as an algebra by its coradical is affine. Motivated by this connection, we investigate the structural properties governing the affineness of these Hopf algebras.

In the case where the Hopf coradical is locally affine and faithfully flat over all its Hopf subalgebras, it follows that the Hopf coradical is affine, which in turn implies that the entire Hopf algebra is affine. This scenario occurs, for example, if the coradical of the Hopf algebra $H$ forms a subalgebra; that is, if $H$ satisfies the dual Chevalley property.

Recall that a Hopf algebra $H$ is called \textit{locally affine} if each finite subset of $H$ is contained in an affine Hopf subalgebra of $H$. We record the following important result due to Goodearl and Zhang:

\begin{lemma}\label{lem:GZ2017}\cite{GZ2017}
Let $H$ be a left or right Noetherian Hopf algebra over $\K$. Suppose that $H$ is locally affine and faithfully flat over all its Hopf subalgebras. Then $H$ is affine. Consequently, a left or right Noetherian pointed Hopf algebra is affine.
\end{lemma}

\begin{remark}\label{rmk:conditions-examples}
The conditions required by Lemma~\ref{lem:GZ2017} are satisfied by several prominent classes of Hopf algebras. Specifically, let $H$ be a Hopf algebra over $\K$:
\begin{itemize}
    \item [$\mathrm{(1)}$] $H$ is locally affine whenever:
    \begin{itemize}
        \item [$\mathrm{(i)}$] $S_{H}$ has finite order (which holds, for example, if $H$ is commutative, cocommutative, or finite-dimensional); 
        \item [$\mathrm{(ii)}$] $H$ is pointed (see \cite[Corollary 3.4]{Z2013}; see also \cite[Lemma 4.8]{GZ2017}); or
        \item [$\mathrm{(iii)}$] $H$ is cosemisimple (see \cite[Theorem 3.3]{La1971}; see also \cite[Corollary 10.8.4]{Ra2012}).
    \end{itemize}
    \item [$\mathrm{(2)}$] $H$ is faithfully flat over all its Hopf subalgebras whenever:
    \begin{itemize}
        \item [$\mathrm{(i)}$] $H$ is commutative or cocommutative (see \cite[Theorem 3.1]{Ta1972}); 
        \item [$\mathrm{(ii)}$] $H$ is pointed or finite-dimensional (see \cite{Ra1977, NZ1989}; see also \cite[Theorem 9.3.1]{Ra2012}); or
        \item [$\mathrm{(iii)}$] $H$ is cosemisimple (see \cite[Theorem 2.1]{Ch2014}).
    \end{itemize}
\end{itemize}
\end{remark}

Utilizing Lemma~\ref{lem:GZ2017} and standard structural results, we deduce that any Hopf subalgebra of a Noetherian cosemisimple Hopf algebra inherits the property of being affine:

\begin{corollary}\label{cor:Noetherian-cosemisimple-Hopfalg-affine}
Let $H$ be a cosemisimple Hopf algebra over $\K$. If $H$ is left or right Noetherian, then any Hopf subalgebra $T$ of $H$ is both Noetherian and affine.
\end{corollary}

\begin{proof}
Assume without loss of generality that $H$ is left Noetherian. By \cite[Theorem 2.1]{Ch2014}, $H$ is faithfully flat over $T$. It then follows from \cite[Exercise 17T]{GW2004} that $T$ is left Noetherian. Evidently, $T$ inherits cosemisimplicity from $H$. By \cite[Theorem 3.3]{La1971}, the antipode $S_{H}$ is bijective and satisfies $S_{H}^{2}(C) = C$ for any simple subcoalgebra $C$ of $T$. Hence, $T$ is Noetherian and locally affine. Since $T$ is faithfully flat over all its Hopf subalgebras, Lemma~\ref{lem:GZ2017} implies that $T$ is affine. An analogous argument applies if $H$ is right Noetherian.
\end{proof}

Combining Theorem~\ref{thm:main-theorem4} and Corollary~\ref{cor:Noetherian-cosemisimple-Hopfalg-affine}, we obtain the following structural criterion:

\begin{corollary}\label{cor:when-Hopfalg-is-affine-2}
A left or right Noetherian Hopf algebra $H$ over $\K$ is affine provided that $H_{(0)}$ is a subalgebra of $H$; that is, $H$ satisfies the dual Chevalley property.
\end{corollary}

Combining Theorem~\ref{thm:main-theorem4} and Lemma~\ref{lem:GZ2017}, we obtain general sufficient conditions for the affineness of Noetherian Hopf algebras based on their first filtration step:

\begin{theorem}\label{thm:when-Hopfalg-is-affine}
Let $H$ be a left or right Noetherian Hopf algebra over $\K$. Suppose that the Hopf coradical $H_{[0]}$ is locally affine and faithfully flat over all its Hopf subalgebras. Then $H_{[0]}$ is affine, and hence $H$ is affine.
\end{theorem}

\begin{proof}
By Lemma~\ref{lem:Sk2006}, the Hopf coradical $H_{[0]}$ forms a left or right Noetherian Hopf subalgebra of $H$. Lemma~\ref{lem:GZ2017} then guarantees that $H_{[0]}$ is affine. It follows immediately from Theorem~\ref{thm:main-theorem4} that the entire Hopf algebra $H$ is affine.
\end{proof}

\bibliographystyle{plain}

\end{document}